\documentclass[12pt]{amsart}%
\usepackage{amsmath}
\usepackage{graphicx}
\usepackage[small,nohug,heads=littlevee]{diagrams}
\usepackage{amscd}
\usepackage{geometry}
\usepackage{amsfonts}
\usepackage{amssymb}%
\newtheorem{theorem}{Theorem}[section]
\theoremstyle{plain}

\newtheorem{claim}{Claim}

\newtheorem{lemma}[theorem]{Lemma}

\newtheorem{remark}{Remark}

\numberwithin{equation}{section}

\let\oldtocsection=\tocsection
\let\oldtocsubsection=\tocsubsection
\let\oldtocsubsubsection=\tocsubsubsection
\renewcommand{\tocsection}[2]{\hspace{0em}\oldtocsection{#1}{#2}}
\renewcommand{\tocsubsection}[2]{\hspace{2em}\oldtocsubsection{#1}{#2}}
\renewcommand{\tocsubsubsection}[2]{\hspace{4em}\oldtocsubsubsection{#1}{#2}}

\usepackage{graphicx}

\usepackage{amsmath}

\usepackage{amsfonts}

\usepackage{amsthm}

\usepackage{amssymb}

\usepackage{latexsym}
\usepackage{fancyhdr}

\usepackage[small,nohug,heads=littlevee]{diagrams}

\usepackage{wrapfig}

\usepackage{enumerate}
\usepackage{soul}

\newtheorem{thm}{Theorem}

\newtheorem{cor}[thm]{Corollary}
\newtheorem{prop}[thm]{Proposition}
\theoremstyle{definition}

\newtheorem{defn}[thm]{Definition}
\newtheorem{ques}[thm]{Question}

\newtheorem{note}[thm]{Note}

\numberwithin{thm}{section}

\newcommand{\intr}{\text{int}}

\newcommand{\bdrysum}{\diamond}

\newcommand{\cl}{\text{cl}}

\newcommand{\mb}{\mathbb}

\newcommand{\Bd}{\text{Bd}}
\newcommand{\donut}{S^1 \times \mathbb{B}^3}

\begin{document}
\title[The Double $n$-Space Property for Contractible $n$-Manifolds]{The Double $n$-Space Property for Contractible $n$-Manifolds}
\author{Pete Sparks}
\address{Marquette University,
Milwaukee, Wisconsin 53201}
\email{peter.sparks@marquette.edu}
\thanks{Material in this article comes from the author's Ph.D thesis which was completed at the University of Wisconsin-Milwaukee, under the supervision of Craig Guilbault.}

\date{September 27, 2016}

\begin{abstract}
Motivated by a recent paper of Gabai \cite{Gab} on the Whitehead contractible 3-manifold, we investigate contractible manifolds $M^n$ which decompose or \emph{split} as $M^n = A \cup_C B$ where $A,B,C \approx \mathbb{R}^n$ or $A,B,C \approx \mathbb{B}^n$. Of particular interest to us is the case $n=4.$ Our main results exhibit large collections of $4$-manifolds that split in this manner.

\end{abstract}
\maketitle

\section{Introduction}

Our results will generally be in the topological category but because of the niceness of the spaces involved we are able to work in both the piecewise linear and smooth categories in our effort to obtain them. \cite{RoSa} is a good source for the piecewise linear theory we will employ. 
\begin{defn} We will write $A \cup_C B$ to indicate a union $A \cup B$ with intersection $C.$ We say a manifold $M^n$ \emph{splits} if $M^n = A \cup_C B$ with $A,B,$ and ${C \approx \mathbb{B}^n}$ or $A,B,$ and $C\approx \mathbb{R}^n$. In the former case we say $M$ ``splits into closed balls" or $M$ is a ``closed splitter" and write $M^n = \mathbb{B}^n \cup_{\mathbb{B}^n} \mathbb{B}^n.$ In the latter case we say $M$ ``splits into open balls" or $M$ is an ``open splitter" and write $M^n = \mathbb{R}^n \cup_{\mathbb{R}^n} \mathbb{R}^n.$
\end{defn}

We are interested in contractible manifolds $M^n$ which are open or closed splitters. We introduce a 4-manifold $M$ containing a spine, which we call a Jester's Hat, that can be written as $A \cup_C B$ with $A,B,$ and $C$ all collapsible. We'll show that this implies $M$ is a closed splitter. Using $M$ as a model we obtain a countably infinite collection of distinct 4-manifolds all of which are closed splitters. 

\begin{thm}
\label{infinite_closed_splitters}
There exists an infinite collection of topologically distinct splittable compact contractible 4-manifolds. The interiors of these are topologically distinct contractible splittable open 4-manifolds.
\end{thm}

By combining the above examples with an infinite connected sum operation, we will then prove the following.

\begin{thm} There exists an uncountable collection of contractible open 4-manifolds which split as $\mathbb{R}^4 \cup_{\mathbb{R}^4} \mathbb{R}^4.$ 
\label{uncountable open 4-splitters}
\end{thm}

Our motivation comes from David Gabai's result that the Whitehead 3-manifold, $Wh^3,$ splits into open 3-balls
\[Wh^3=  \mathbb{R}^3 \cup_{\mathbb{R}^3} \mathbb{R}^3 \ \ \text{ \cite{Gab}}.\]
Other terminology in use which is synonymous with open splitting includes \textit{double $n$-space property} and \textit{Gabai splitting}. Garity, Repov$\check{\text{s}}$, and Wright have recently discovered uncountable collections of both 3-dimensional contractible open splitters and 3-dimensional contractible nonsplitters (see Theorems \ref{uncountable open 3-splitters} and \ref{uncountable open 4-nonsplitters} below)\cite{GRW}. 

\section{Background and History}
\subsection{Elementary Results}
It is clear that the unit ball $\mathbb{B}^n$ splits into two ``subballs" overlapping in a $n$-ball. Likewise, Euclidean space splits into two Euclidean spaces meeting in a Euclidean space. More generally, we have the following (which was assumed without proof in \cite{Gla65}). A proof can be found in \cite[Prop ~1.2.1]{Spa}.

\begin{prop} If $M^n$ splits as ${M^n = \mathbb{B}^n \cup_{\mathbb{B}^n} \mathbb{B}^n}$ then 
$\intr M^n$ splits as $\intr M^n = \mathbb{R}^n \cup_{\mathbb{R}^n} \mathbb{R}^n$.
\label{ints of splitters split}
\end{prop}

\subsection{History and Current Work}

Some classical knowledge about manifold splitting was provided by Glaser.
\begin{thm}
(a) For each $n\geq 4$ there exists a compact contractible PL $n$-manifold with boundary $W^n$ not homeomorphic to $\mathbb{B}^n$ such that ${W^n \approx \mathbb{B}^n \cup_{\mathbb{B}^n} \mathbb{B}^n.}$

(b) For each $n \ge 3$ there exist an open contractible $n$-manifold $O^n$ not homeomorphic to $\mathbb{R}^n$ such that $O^n \approx \mathbb{R}^n \cup_{\mathbb{R}^n} \mathbb{R}^n.$

\end{thm}

For the compact case, Glaser shows the existence of a contractible $(n-2)$-complex piecewise linearly embedded in $S^n$ which has a non-ball regular neighborhood which splits. The $n\geq5$ case was shown in \cite{Gla65} and the $n=4$ case was shown in \cite{Gla66}.

For the noncompact $n \ge 4$ case he takes the interiors of the compact splitters found in (a). For the noncompact $n=3$ case, Glaser shows that the complement of a certain embedding of a double Fox-Artin arc in $S^3$ splits and is not a (open) ball \cite{Gla66}. 

In \cite{Gab}, Gabai asks
\begin{ques}
Is there a reasonable characterization of open contractible 3-manifolds that are the union of two embedded submanifolds each homeomorphic to $\mathbb{R}^3$ and that intersect in a $\mathbb{R}^3$?
\end{ques}
Renewed interest in this topic, motivated by Gabai's splitting of the Whitehead manifold and the resulting above question, has led to the following recent results.

\begin{thm}(see \cite{GRW}) 
There exist uncountably many distinct contractible 3-manifolds that are open splitters.
\label{uncountable open 3-splitters}
\end{thm}

\begin{thm} (see \cite{GRW}) 
There are uncountably many distinct contractible 3-manifolds that are not open splitters.
\label{uncountable open 4-nonsplitters}
\end{thm}

\begin{note} In dimension 3, the Poincar$\acute{\mbox{e}}$ conjecture gives that every compact contractible manifold is homeomorphic to $\mathbb{B}^3$ so the question of closed splitters in this case is uninteresting.
\end{note}

Earlier work by Ancel-Guilbault \cite{AG95} and more recently by Ancel-Guilbault-Sparks \cite{AGS} provides a great deal of information about splitters in dimensions greater than or equal to 5.

\begin{thm} 
If $C^n$ $(n \geq 5)$ is a compact, contractible $n$-manifold then $C^n$ splits as $\mathbb{B}^n \cup_{\mathbb{B}^n} \mathbb{B}^n.$ 
\end{thm}

\begin{cor} 
For $n \geq 5:$
	\begin{enumerate}
  	\item the interior of every compact contractible $n$-manifold is an open splitter, and
  	\item there are uncountably many non-homeomorphic $n$-manfolds which are open splitters.
	\end{enumerate}
\end{cor}

\begin{thm} 
For $n \geq 5,$ every Davis $n$-manifold is an open splitter.
\end{thm}

\section{The Mazur and Jester's Manifolds}
\subsection{The Mazur Manifold}
\label{sec:MazMan}

\begin{figure}[h!]
\centering
\vspace{-6.7in}
\includegraphics[height=8in]{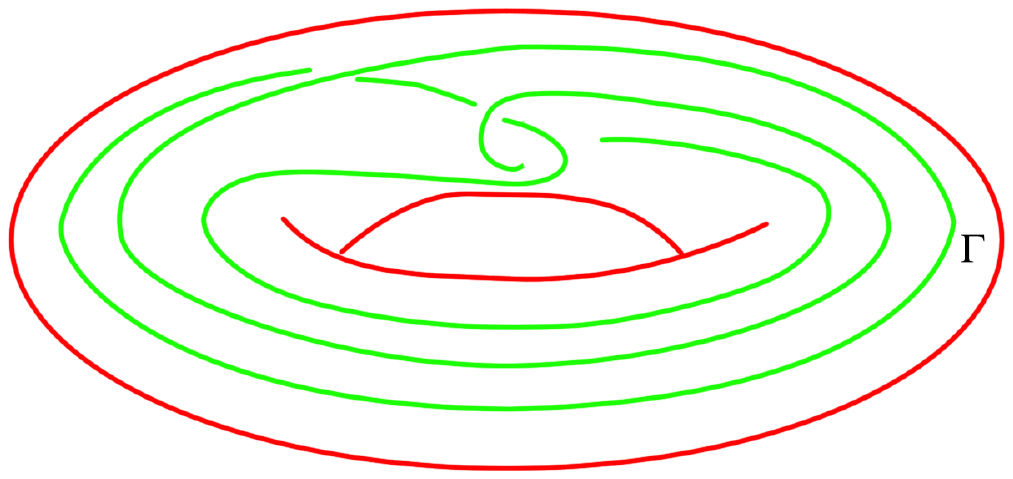}                   
\caption{$\Gamma \subset \partial( S^1 \times \mathbb{B}^3)$} 
\label{Mazcurve}
\end{figure}

In \cite{Maz}, Barry Mazur described what are now often called \emph{Mazur manifolds}.  
Starting with a $S^1\times\mathbb{B}^3$ one adds a 2-handle ${h^{(2)}\approx \mathbb{B}^2 \times \mathbb{B}^2}$ along the curve $\Gamma$ shown in Figure \ref{Mazcurve}. 
That is,   
\[M\!a^4_\Phi = S^1\times\mathbb{B}^3\cup_\Phi \mathbb{B}^2 \times \mathbb{B}^2\]
is a \emph{Mazur manifold}. Here $\Phi$ is a framing $\Phi:S^1\times\mathbb{B}^2\rightarrow T_{\Gamma}$, $T_{\Gamma}$ is a tubular neighborhood of $\Gamma$ in $\partial(\donut)$ and the domain $S^1\times\mathbb{B}^2$ is the first term in the union \[S^1\times\mathbb{B}^2\cup \mathbb{B}^2\times S^1=\partial(\mathbb{B}^2\times \mathbb{B}^2)\]

For each Dehn twist of the $S^1 \times S^1 = \partial(S^1\times \mathbb{B}^2)$ sending $S^1 \times p$ $(p \in S^1)$ to a closed curve (that is, an integer number of full twists), there exists a framing $\Phi.$ Thus the number of framings is infinite. Mazur chose a specific framing $\varphi$ yielding a specific manifold, which we'll denote $M\!a^4,$ for which he showed $\partial M\!a^4 \not\approx S^3$ so $M\!a^4 \not\approx \mathbb{B}^4.$ The chosen framing corresponds to a parallel copy of $\Gamma$ say $\Gamma'=\varphi(S^1 \times p)$ which lies at the ``top" (the up direction is perpendicular to the page, toward the viewer) of $S^1 \times \mathbb{B}^2.$ Thus there are no twists with this framing.

\begin{figure}[!ht]
\includegraphics[trim=0 250 0 200,height=3.5in, clip]{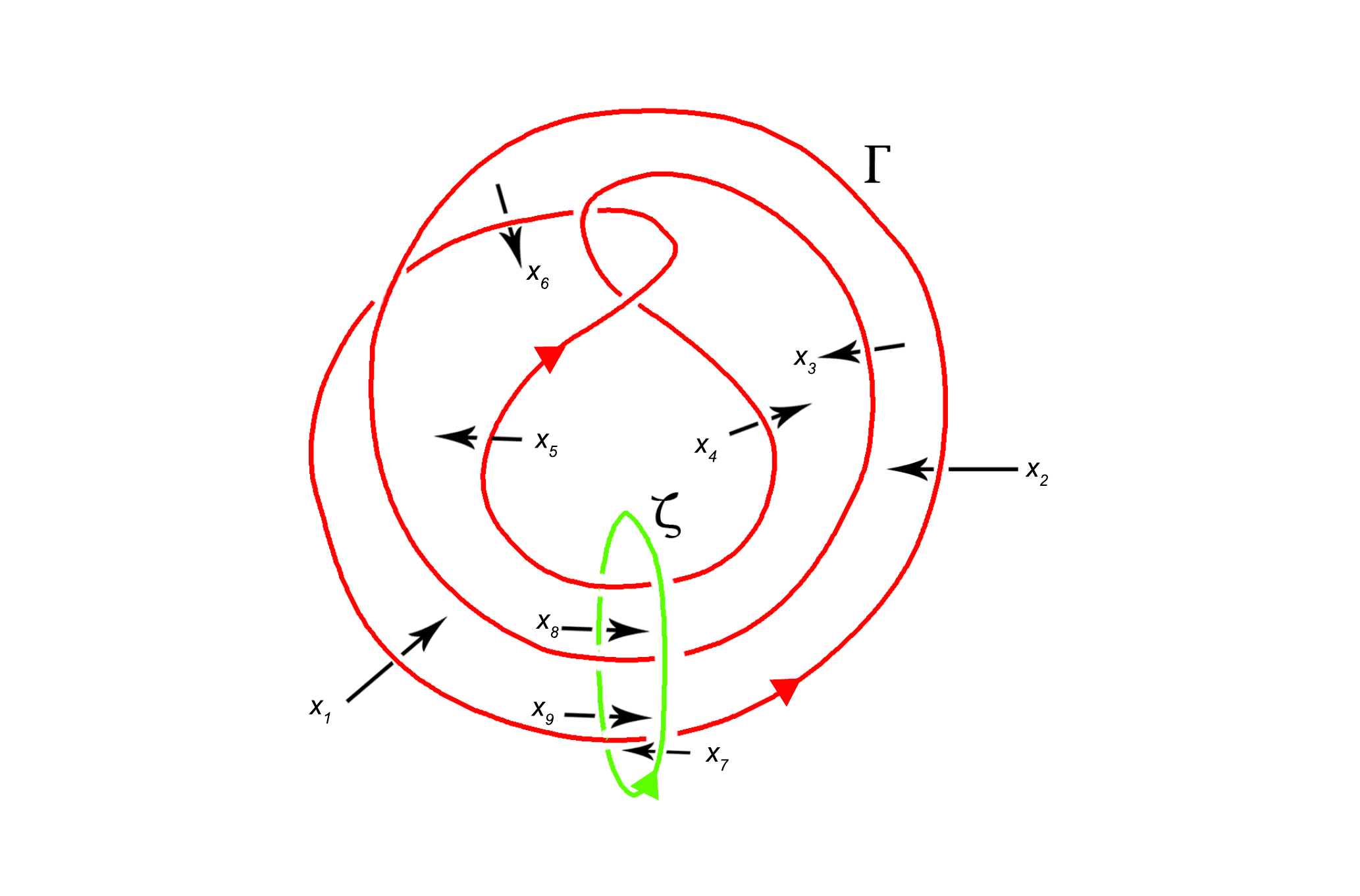}                   
\centering
\caption{Wirtinger diagram of the Mazur link}
\label{Mazurlink}
\end{figure}

Here we'll describe our interpretation of his argument for the nontriviality of $\pi _1(\partial M\!a^4).$ The details of this calculation will play a key role in our proof of Theorem \ref{infinite 4-splitters}. Starting with the link $\Gamma \cup \zeta$ in $S^3$ pictured in Figure \ref{Mazurlink}, we obtain said figure's Wirtinger presentation (see \cite[p. 56]{Rol} for a treatment of Wirtinger presentations). This gives a presentation with exactly one generator for each arc in the link diagram. These generators correspond to the loops in $S^3$ which start at the viewer's nose (the basepoint), travel under the arc, and then return home (to the nose). Thus in our picture the generators are the $x_i$ as pictured. The relators in the presentation correspond to the undercrossings of pairs of arcs. As there are 9 undercrossings the Wirtinger presentation of this link diagram has 9 generators and 9 relators: $\left\langle x_1,...,x_9|r_1,...,r_9\right\rangle.$ 
We then perform a Dehn drilling on a tubular neighborhood, $N(\zeta)\approx \mathbb{B}^2\times S^1,$ of $\zeta.$ That is, we remove $\intr N(\zeta).$ Next, we perform a Dehn filling by sewing in $N(\zeta)$ backwards (ie sewing in a $S^1 \times \mathbb{B}^2$) along  $\partial N(\zeta).$ This Dehn surgery on ${S^3\approx (S^1\times \mathbb{B}^2) \cup_{S^1 \times S^1} (\mathbb{B}^2 \times S^1)}$ results in an $(S^1\times \mathbb{B}^2) \cup_{S^1 \times \partial \mathbb{B}^2} (S^1 \times \mathbb{B}^2)\approx S^1 \times S^2 $ with $\Gamma$ embedded as in Figure \ref{Mazcurve}. This surgery exchanges $N(\zeta)$'s meridian with its longitude. Thus the group element corresponding to following around $\zeta$ 
is killed and we must add in a relator, say $r_\zeta=x_5x_2^{-1}x_1^{-1}=1,$ to our presentation to adjust for this.

Adding a 2-handle along $\Gamma$ (and throwing out its portion of $M\!a^4$'s interior) gives our $\partial M\!a^4=(S^1 \times S^2 -\intr N(\Gamma))\cup_{\partial N(\Gamma)} (\mathbb{B}^2 \times S^1).$ We describe the gluing of $\mathbb{B}^2 \times S^1$ in two steps. We first glue in a thickened meridional disc, $D,$ which kills off $\Gamma'$ the curve to which it is it is attached. 
 Thus to our Wirtinger presentation we introduce a relator $r_\Gamma=x_7^{-1}x_5^{-1}x_7x_3^{-1}x_2^{-1}x_7^{-1}=1.$ We next glue on the rest of $\mathbb{B}^2 \times S^1.$ The closed complement of $D$ in $\mathbb{B}^2 \times S^1$ is a 3-ball and it is attached along its entire boundary. Adding such does not change the fundamental group and thus $\pi_1(\partial M\!a^4)\cong  \left\langle x_1,...,x_9|r_1,...,r_9,r_\zeta,r_\Gamma\right\rangle.$

Proceeding as in \cite{Maz}, let $\beta = x_7,$ $\lambda=x_2,$ (see Fig.  \ref{Mazurlink}) and $\alpha=\beta \lambda.$ Via Tietze transformations (see \cite [p.~79]{Geo} for a treatment on Tietze transformations), it was shown in \cite{Maz} that 
	\[\pi_1(\partial M\!a^4) \cong <\alpha, \beta | \beta^5=\alpha^7, 				\beta^4=\alpha^2\beta\alpha^2> \text{ and }\]
	\[G:=\pi_1(\partial M\!a^4)/\text{nc}\{\beta^5=1\} \cong <\beta,\gamma|\gamma^7=\beta^5=(\beta\gamma)^2=1>\]
where $\gamma=\alpha^2.$  
We claim $G$ maps nontrivially into the subgroup of the isometries of the hyperbolic plane generated by reflections in the geodesics containing the edges of a triangle with angles $\pi/7,$ $\pi/5,$ and $\pi/2.$ 
That is, there exists a homomorphism
 \[h:G \rightarrow \text{Isom}(\mb{H}^2)\]
so that $\text{Im} h$ can be generated by rotations with centers at the vertices of a triangle  $\Delta ABC$ with angles $\pi/7,$ $\pi/5,$ and $\pi/2.$ 
 Here $h(\beta)=$ rotation with angle $-2\pi/5$ at $C$ and $h(\gamma)=$ rotation with angle $2\pi /7$ at $A$. 

We'll show the relator $h((\beta \gamma)^2)=1$ is satisfied. Let $r_{XY}$ be reflection in the geodesic containing $X$ and $Y.$ Then $h(\beta)=r_{BC}\circ r_{AC}$ and  $h(\gamma)=r_{AC}\circ r_{AB},$ so that $h(\beta)h(\gamma)= r_{BC}\circ r_{AC}\circ r_{AC}\circ r_{AB}=r_{BC} \circ r_{AB}.$ This last isometry is a rotation at $B$ with angle $-\pi$ and $h(\beta \gamma)$ is shown to have order 2.

This shows $\text{Im} h$ is nontrivial. Hence $\pi_1(\partial M\!a^4)$ is nontrivial and thus $\partial M\!a^4 \not\approx S^3.$

We now state and prove the following Proposition which we will employ in Subsection \ref{Theorem of Wright}.

\begin{prop}
\label{gamma nontrivial}
Let $m_\Gamma$ be the meridian of the torus $\partial T_\Gamma.$ Then $m_\Gamma$ is nontrivial in $S^1 \times S^2 - \text{int}(T_\Gamma).$
\end{prop}

\begin{proof}
We choose $x_5$ as our representative of $m_\Gamma.$ By the relator \[r_9:x_1=x_7^{-1}x_2x_7=\beta^{-1}\lambda\beta=\beta^{-1}(\beta^{-1}\alpha)\beta\] we get $x_1=\beta^{-2}\alpha\beta.$ By $r_\zeta: x_5=x_1x_2$ we obtain \[x_5=(\beta^{-2}\alpha\beta)(\beta^{-1}\alpha)=\beta^{-2}\alpha^2=\beta^{-2}\gamma.\]

Thus 
\begin{eqnarray}
h(x_5)&=&h(\beta^{-2}\gamma)\nonumber \\
			&=& h(\beta^{-2})h(\gamma)\nonumber\\
			&=&	\text{(rotation of }4\pi/5 \text{ at } C)\text{(rotation of }2\pi/7 \text{ at } A)\nonumber\\
			&\neq& 1_{\mb{H}^2} \nonumber \hspace{1cm}\mbox{(since $A$ is not fixed).}
\end{eqnarray}
Thus $x_5$ is not trivial in $\partial M\!a^4.$ Hence $x_5$ is nontrivial in $S^1 \times S^2- \text{int}(T_\Gamma).$ This concludes the proof of Proposition \ref{gamma nontrivial}. 
\end{proof}

The following question is still open.

\begin{ques} 
\label{Ma Split?}
Does $M\!a^4$ split into closed balls?
\end{ques}

\begin{ques} 

Does there exist an infinite number of closed 4-dimensional splitters?
\end{ques}

We will give an answer to this question in Subsection \ref{Theorem of Wright}.

\subsection{The Jester's Manifolds}
As an initial step towards constructing 4-dimensional splitters, we describe a collection of 4-manifolds similar to Mazur's. 
Start with a $S^1 \times \mathbb{B}^3$ and within its $S^1 \times S^2$ boundary select a curve $C$ as follows. Let $T$ be a tubular neighborhood of $C$ in our $S^1 \times S^2$. We have chosen $C$ so that it is the preimage of the Mazur curve $\Gamma$ under the standard double covering map ${p:S^1 \times \mathbb{B}^3 \rightarrow S^1 \times \mathbb{B}^3 }.$ 

\begin{figure}[!ht]
\centering
\vspace{-6.5in}
\includegraphics[height=8in]{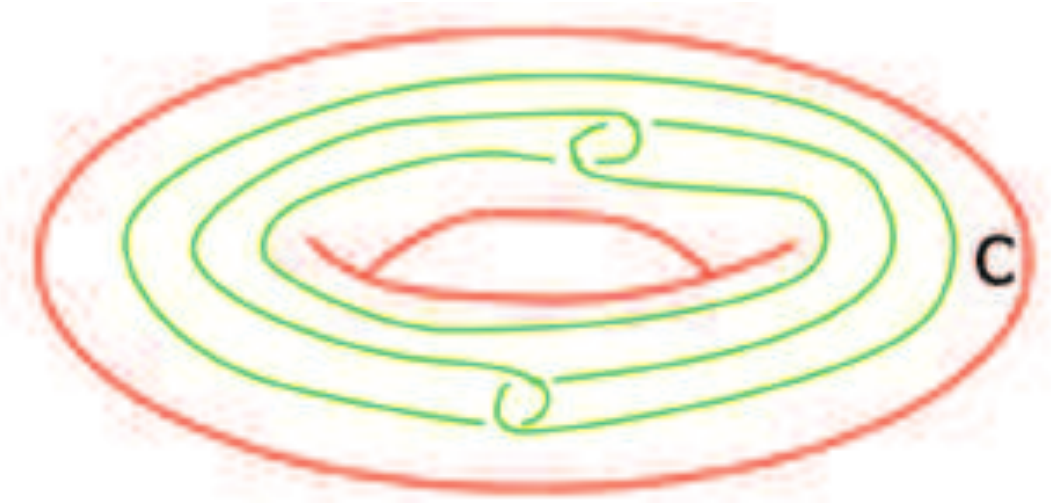}                  
\caption{$C \subset \partial(S^1 \times \mathbb{B}^3)$} 
\end{figure}

Then, given a framing ${\Psi:S^1 \times \mathbb{B}^2 \rightarrow T,}$ define  
\[M_{\Psi}=S^1\times\mathbb{B}^3\cup_{\Psi} \mathbb{B}^2 \times \mathbb{B}^2\]
where the domain is the $S^1\times\mathbb{B}^2$ factor in the boundary of our $2$-handle $h^{(2)}\approx \mathbb{B}^2 \times \mathbb{B}^2.$ We call such an  $M_{\Psi}$ a \emph{Jester's manifold.}

(In Section 4, we will expand our definition of Jester's manifold to include analogous attachments using \emph{pseudo}-handles.)

\begin{remark}Initially, we had hoped that, by altering the framings, we could prove the existence of an infinite collection of these Jester's manifolds. Unfortunately, the group theoretic calculations proved too complicated. Fortunately, however, we were able to get around this problem by employing a technique of David Wright's (see Subsection \ref{Theorem of Wright}). We are still interested in the following question.
\end{remark}

\begin{ques} 
Does there exist a Jester's manifold that is not homeomorphic to a ball? Are there an infinite number of Jester's manifolds (as defined above)?

\end{ques}

\section{Spines}

\subsection{Collapses}

We borrow our definitions 
of collapse from \cite[pp.~3,4,14,15]{Coh73}. We will be denoting the cone over a simplicial complex $A$ with cone point $a$ by $aA.$

\begin{defn}
If $K$ and $L$ are finite simplicial complexes we say that there is an \emph{elementary simplicial collapse} from $K$ to $L,$ and write $K \searrow^e L,$ if $L$ is a subcomplex of $K$ and $K=L \cup aA$ where $a$ is a vertex of $K$, $A$ and $aA$ are simplexes of $K$, and $aA \cap L= a(\partial A).$ We call such an $A$ a \emph{free face} of $K.$
\end{defn}

Observe that a free face completely specifies an elementary simplicial collapse.

\begin{defn} Suppose that $(K,L)$ is a finite CW pair. Then $K \searrow^e L$--i.e. $K$ \emph{collapses to} $L$ \emph{by an elementary collapse}--iff
	\begin{enumerate}
	\item $K=L \cup e^{n-1} \cup e^n$ where $e^n$ and $e^{n-1}$ are not in $L,$
	\item there exists a ball pair $(Q^n, Q^{n-1}) \approx (\mathbb{B}^n, \mathbb{B}^{n-1})$ and a map $\varphi: Q^n \rightarrow K$ such that
		\begin{enumerate}[a)]
		\item $\varphi$ is a characteristic map for $e^n$
		\item $\varphi |Q^{n-1}$ is a characteristic map for $e^{n-1}$
		\item $\varphi (P^{n-1}) \subset L^{n-1},$ where $P^{n-1}\equiv \text{cl}(\partial Q^n - Q^{n-1}).$
		\end{enumerate}
	\end{enumerate}
\end{defn}	

In both the simplicial and CW cases we define

\begin{defn}
$K$ \emph{collapses} to $L$, denoted $K \searrow L,$ if there is a finite sequence of elementary collapses
\[K=K_0\searrow^e K_1 \searrow^e K_2 \searrow^e ... \searrow^e K_l=L.\]
If $K$ collapses to a point we say $K$ is \emph{collapsible} and write $K \searrow 0$.
\end{defn}

\begin{defn} Suppose $M$ is a compact PL manifold. If $K$ is a subcomplex of $M$ contained in $\intr M$ with $M \searrow K$ we say $K$ is a \emph{spine} of $M.$
\end{defn}

We will make use of the following regular neighborhood theory due to J. H. C. Whitehead. The following two propositions, theorem, and corollary can be found in \cite[pp.~40,41]{RoSa}.

\begin{prop} Suppose $M \supset M_1$ are PL $n$-manifolds with $M \searrow M_1.$ Then there exists a homeomorhism $h:M\rightarrow M_1.$ 
\end{prop}

\begin{thm} Suppose $X \subset M,$ where $M$ is a PL manifold, $X$ is  compact polyhedron, and $X \searrow Y.$ Then a regular neighborhood of $X$ in $M$ collapses to a regular neighborhood of $Y$ in $M.$
\end{thm}

Thus if $K$ is a spine of $M$ then for any regular neighborhood $N(K)$ of $K$ in M we have $N(K) \approx M.$

\begin{prop} If $X\searrow 0$ then a regular neighborhood of $X$ is a ball.
\end{prop}

\begin{cor}
Suppose $M$ is a manifold with a spine $K$ and $K\searrow 0.$ Then $M$ is a ball.
\end{cor}

\begin{prop} 
Suppose $W$ is a PL manifold and $A$ and $B$ are simplicial complexes $A,B \subset \intr W.$ If $W\searrow A \cup B$ with $A,B,A\cap B \searrow 0$ then $W$ splits into closed balls.
\end{prop}

\begin{proof} Let $A, B,$ and $C$ be such that $W \searrow A \cup_C B$ with $A,B,C \searrow 0.$ Regular neighborhoods of collapsible subcomplexes are piecewise linear balls. So given a triangulation of $W$ with $A$ and $B$ as subcomplexes, we construct (with respect to this triangulation) regular neighborhoods $N_A$ of $A$ and $N_B$ of $B$ and we have that $N_A$ and $N_B$ are balls and $N_A \cap N_B$ is a regular neighborhood of $C$ and as such is also a ball. $N_A \cup N_B$ is a regular neighborhood of $A \cup B$, a spine of $W$, so $N_A \cup N_B$ is homeomorphic to $W.$
\end{proof}

\subsection{The Dunce Hat}
The \emph{dunce hat}, $D,$ is defined as the quotient space obtained by identifying the edges of a triangular region as pictured in Figure \ref{D}. 

\begin{figure}[!ht]

\vspace{-3in}
\includegraphics[height=5in]{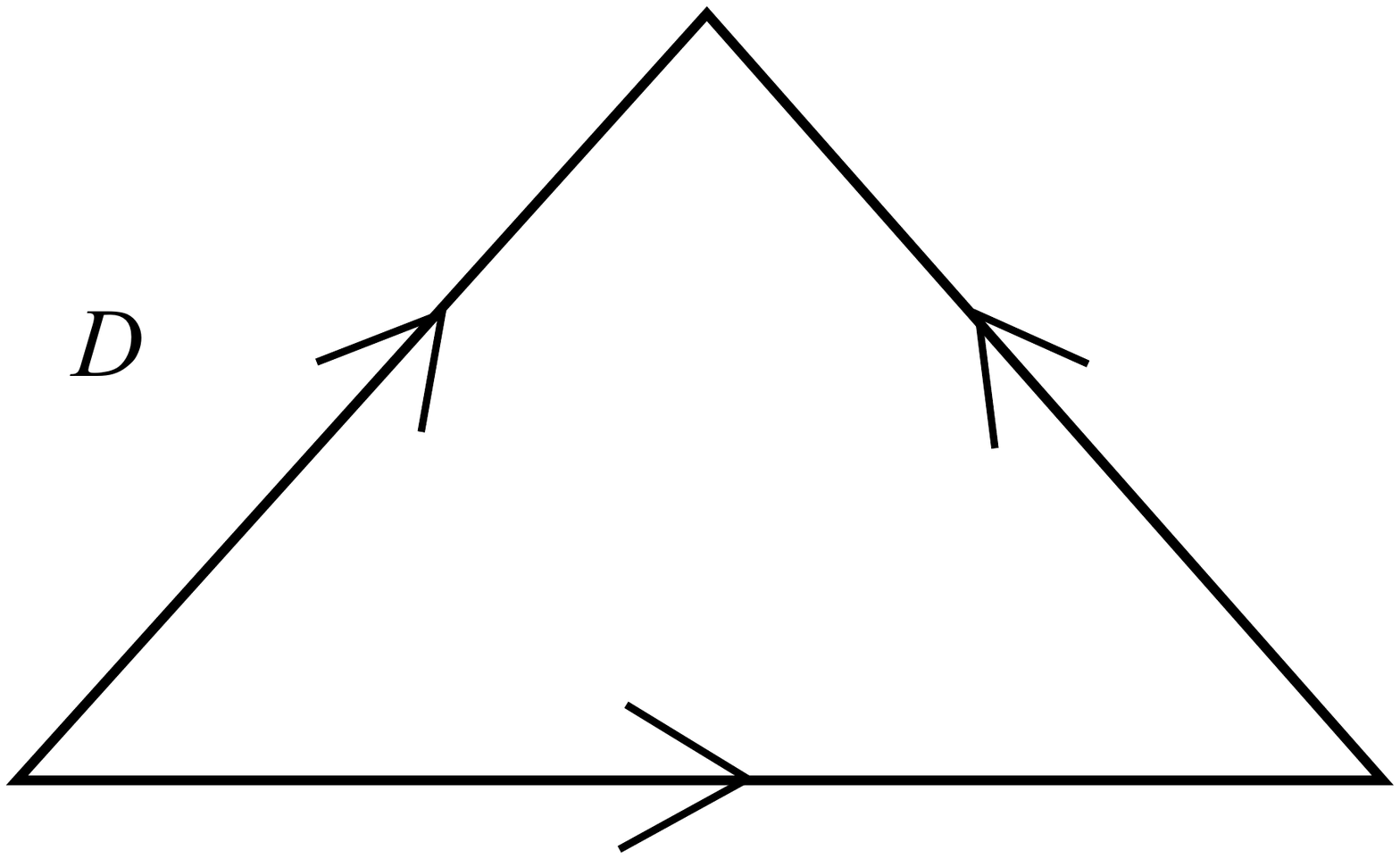}
\centering                 
\caption{The Dunce Hat}
\label{D}
\end{figure}

The dunce hat was one of the first examples of a contractible but not collapsible simplicial complex. 
 A well know result by Zeeman is that the Mazur manifold has a dunce hat spine [Zee]. That observation will become clear in the following subsection, when we identify a spine of a slightly more complicated example.
 
To the best of our knowledge the following question is open.

\begin{ques}
\label{D split?}
Can the dunce hat be expressed as $D=A\cup_{C} B$ with $A,B,C \searrow 0?$ If so, the answer to question \ref{Ma Split?} is yes: $M\!a^4 \approx \mathbb{B}^4 \cup_{\mathbb{B}^4} \mathbb{B}^4.$
\end{ques}

\subsection{The Jester's Hat}
We define the \emph{Jester's hat}, $J$, to be the quotient space obtained from gluing the hexagonal region of the plane as in Figure \ref{J}. 
We can also realize this space by attaching a disc to a circle with the attaching map in Figure \ref{J attaching map}. 
Since the attaching map is homotopic to the identity, $J$ is contractible \cite[p.~16]{Hat}. $J$ is not collapsible as it has no free edge.

\begin{figure}[!ht]
\vspace{-1.5in}
\includegraphics[height=5in]{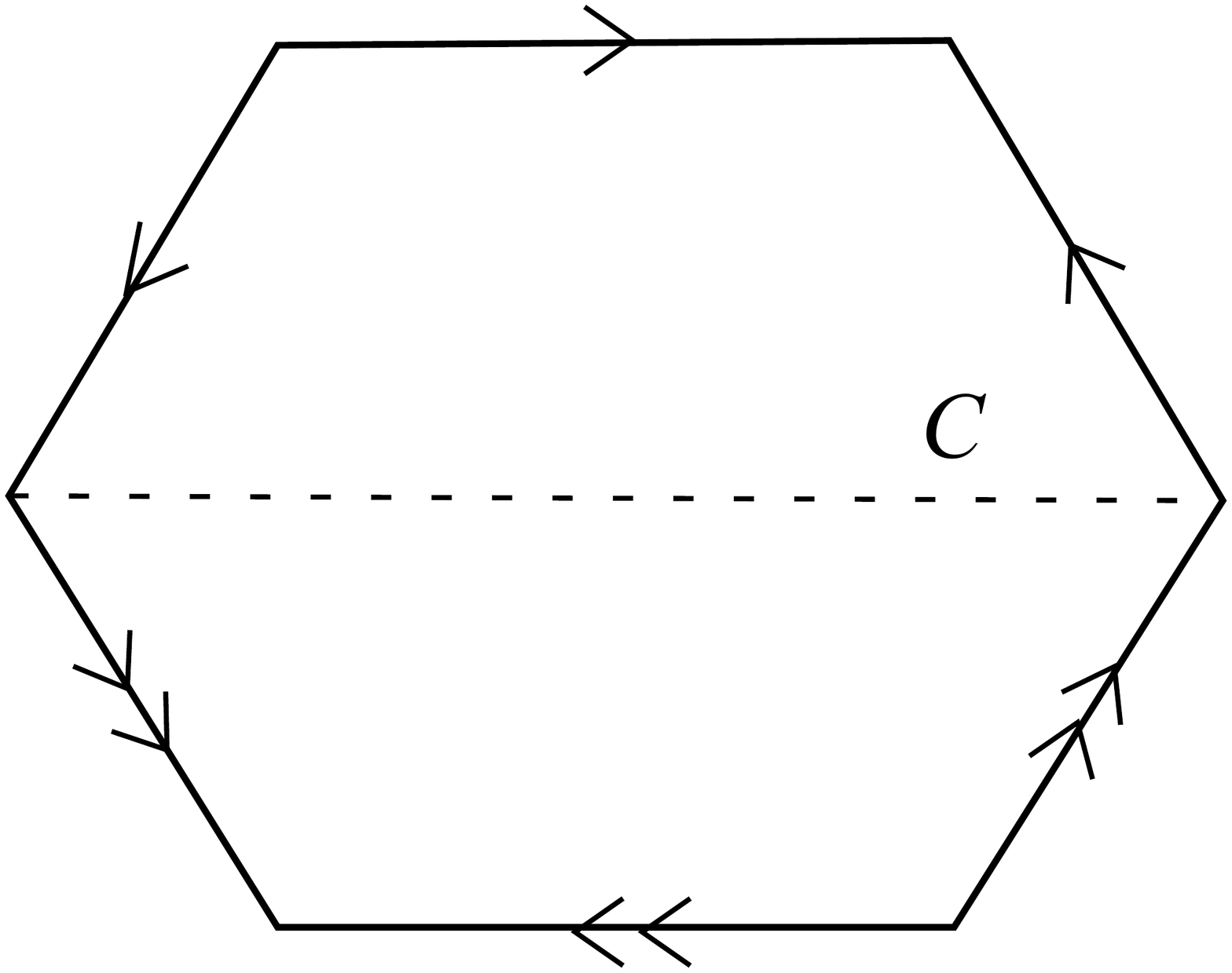} 
\centering                  
\caption{The Jester's Hat}
\label{J}
\end{figure}

\begin{figure}[!ht]
\vspace{-2in}
\includegraphics[height=5in]{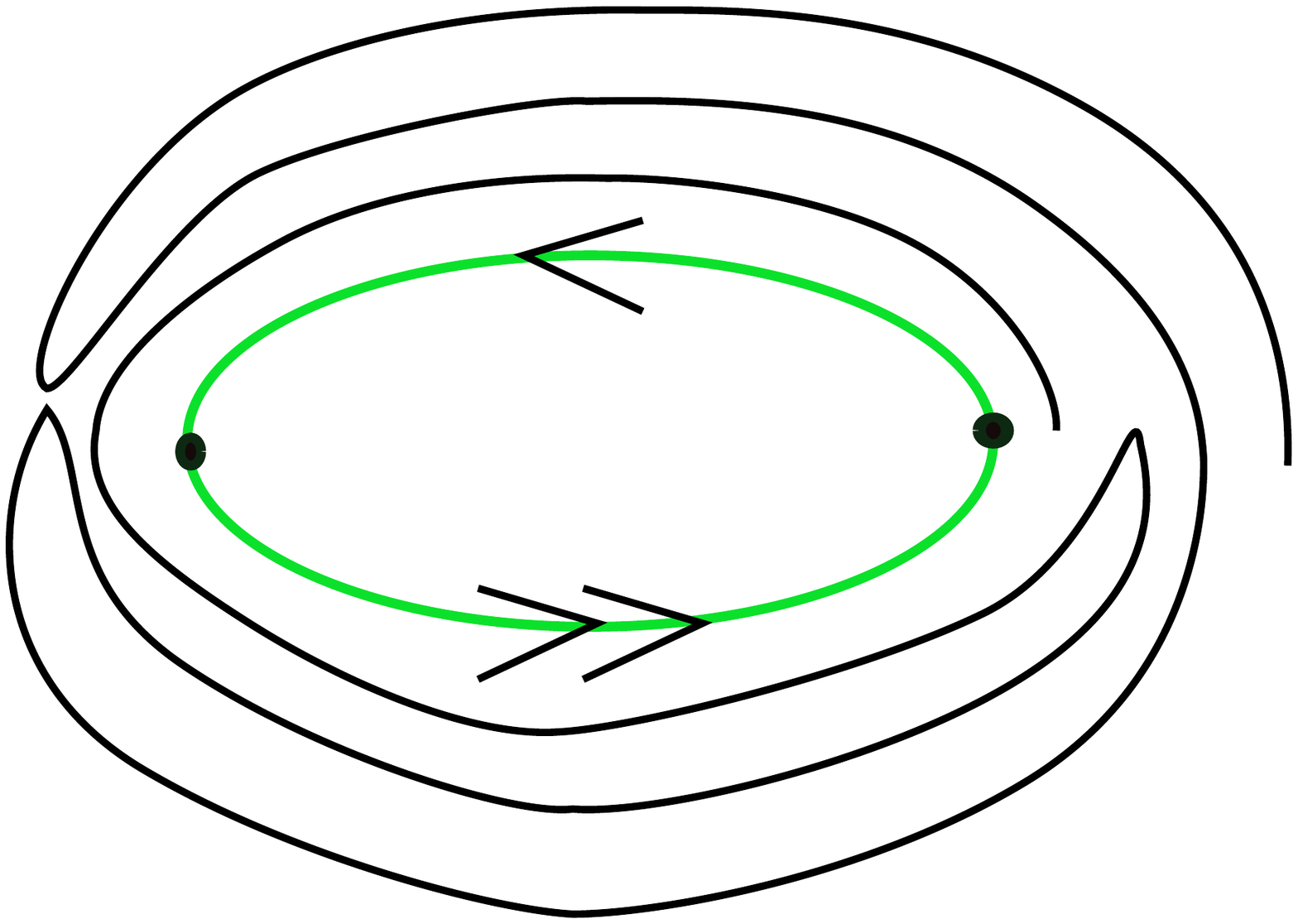} 
\centering                  
\caption{Attaching map for $J$}
\label{J attaching map}
\end{figure}
 
By cutting $J$ open along the dashed arc in Fig.\ref{J}, one can see that $J$ can be decomposed into the union of collapsible subsets intersecting in $C,$ another collapsible subset.
\[J=A\cup_C B \;\text{with} \; A,B,C \searrow 0.\]
 
The interested reader can see \cite[pp.~17-18]{Spa} for details.

\begin{prop}
\label{M collapses to J}

Every Jester's manifold has a Jester's hat spine.
\end{prop}
\begin{proof} The proof is analogous to Zeeman's proof that Mazur's manifold has a dunce hat spine \cite{Zee}. Let $M=M_\Psi$ be a Jester's manifold for a given framing $\Psi.$
We divide the $S^1$ of the $S^1 \times S^2$ in which $C$ resides into four arcs $I_1,I_2,I_3,$ and $I_4$ so that $I_1 \times S^2$ and $I_2 \times S^2$ each contain a ``clasp" of $C$ (see Figure \ref{intervals and their clasps}).
 
\begin{figure}[!ht]
\centering 
\vspace{-5.5in}
\includegraphics[height=8in]{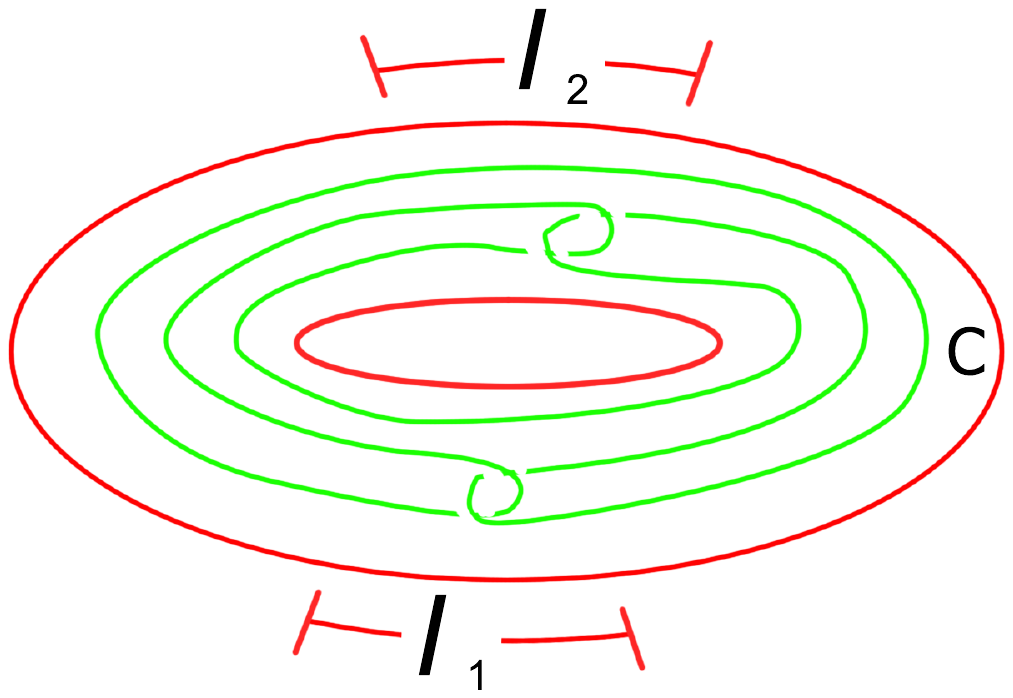}
                

\caption{Intervals of $S^1$ and their clasps}
\label{intervals and their clasps}
\end{figure}

For $i=1,2$, let $f_i:S^1 \rightarrow S^1$ be the map that shrinks $I_i$ to a point, say $p_i,$ and is a homeomorphism on the complement of $I_i.$ 
Further let $\pi:S^1 \times S^2 \rightarrow S^1$ be projection onto the first factor, $j$ be the inclusion  $C \hookrightarrow S^1 \times S^2$, ${g=f_1\circ f_2 \circ \pi:S^1 \times S^2 \rightarrow S^1}$ and $h=g \circ j.$ Let $M(g)$ and $M(h)$ be the mapping cylinders of $g$ and $h$, respectively. That is, 
\[M(g)=[(S^1 \times S^2 \times [0,1]) \sqcup S^1]/\sim_g \ \ \text{and} \ \
 M(h)=[(C \times [0,1]) \sqcup S^1]/\sim_h\]
where $\sim_g$ and $\sim_h$ are generated by $(x,0) \sim_g g(x)$ and $(y,0) \sim_h h(y)$, respectively.  

$M(g)$ is homeomorphic to $S^1 \times \mathbb{B}^3$ \cite[p.~19]{Spa}. Since $h=g|_C$, $M(h)$ is a subcylinder of $M(g)$ and by a result of J.H.C. Whitehead $M(g) \searrow M(h)$  \cite{Whi}. Further, the 2-handle $h^{(2)}$ viewed as $\mathbb{B}^2 \times \mathbb{B}^2$ in our construction of $M$ collapses onto its core union the attaching tube: ${(\mathbb{B}^2 \times \{0\})\cup(S^1 \times \mathbb{B}^2).}$ 
Follow this with the collapse of $M(g)$ onto $M(h)$ to obtain the collapse:
 
\[M = S^1 \times \mathbb{B}^3 \cup_{\Psi} \mathbb{B}^2 \times \mathbb{B}^2 \searrow S^1 \times \mathbb{B}^3 \cup_\Psi [(\mathbb{B}^2 \times \{0\})\cup(S^1 \times \mathbb{B}^2)] \searrow M(h) \cup_{\Psi |^C} B^2.\] 
But from the illustration of $M(h)$ (Figure \ref{M(h)}) we can see that $M(h) \cup_{\Psi |^C} B^2$ is our Jester's hat $J$.
 
\end{proof}

\begin{figure}[!ht]
\vspace{-2.5in}
\includegraphics[height=5in]{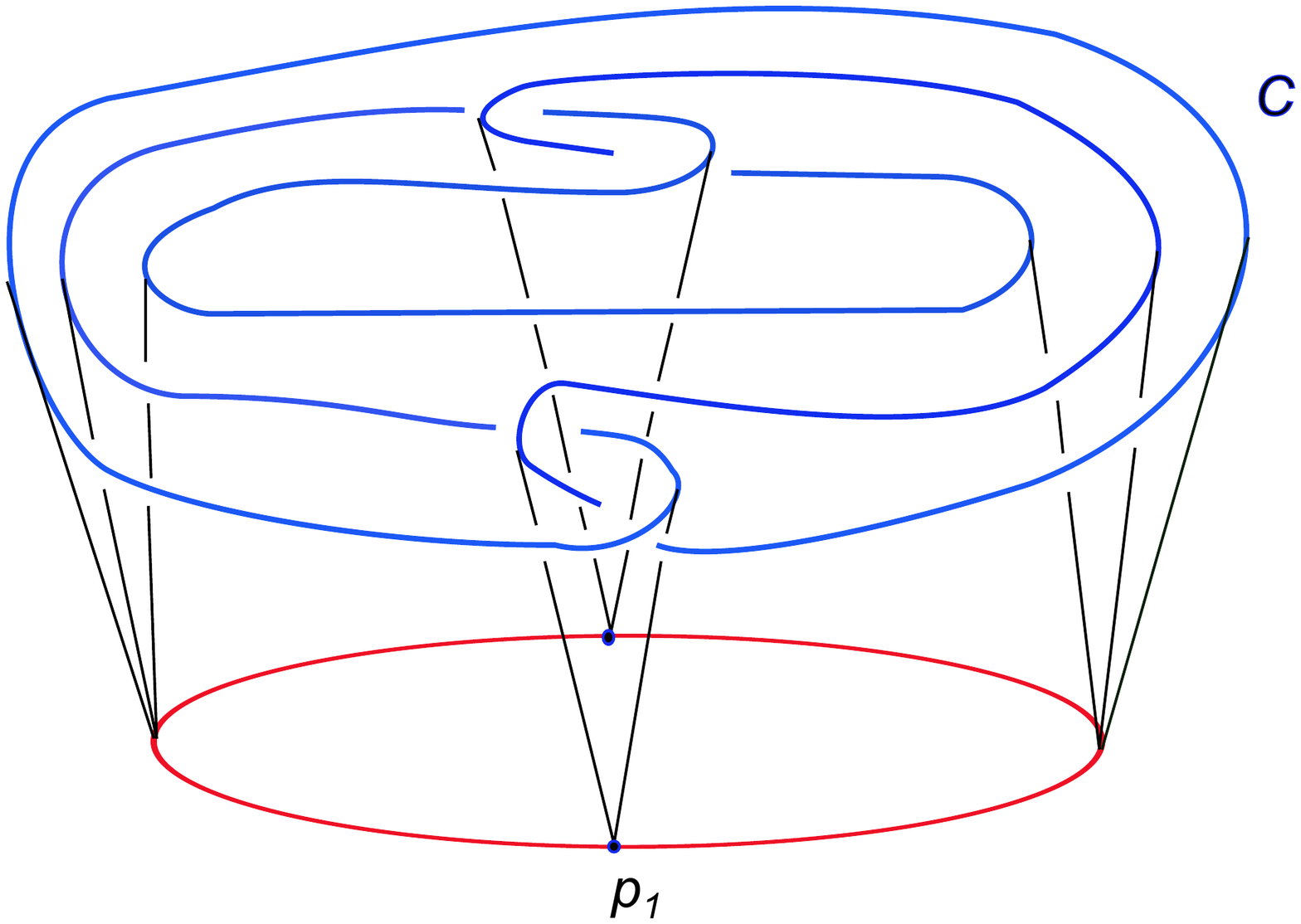}                   
\centering
\caption{The Mapping Cylinder of $h$}
\label{M(h)}
\end{figure}

\begin{cor} The Jester's manifolds split into closed $4$-balls.
\end{cor}

\begin{remark} While we now know that the $M_\Psi\emph{'s}$ split into closed balls, we have not demonstrated that any $M_\Psi$ is not just a ball. To deal with that issue we will modify the construction.
\end{remark}

\section{More Jester's Manifolds}

For this section we let $M=M_\Psi$ be an arbitrary Jester's manifold. Recall $\Psi$ is the framing $\Psi: S^1 \times \mathbb{B}^2 \rightarrow T$ and $T$ is a tubular neighborhood of the curve $C$ in $\partial (S^1 \times \mathbb{B}^3).$

\subsection{Pseudo 2-handles} 

Using $M$ as a model, we apply a construction due to Wright to obtain 
a collection of manifolds $\{W_i\},$ as follows \cite{Wri}. To construct $W_i,$ we start with the $S^1 \times \mathbb{B}^3$ of the Jester's manifold construction and attach a ``pseudo 2-handle", a $\mathbb{B}^4,$ along $K_i,$ the connected sum of $i$ trefoils in the boundary of $\mathbb{B}^4,$ to the curve $C$ in $\partial(S^1 \times \mathbb{B}^3).$ (See Figure \ref{pseudo}.) That is, 
\[W_i = S^1 \times \mathbb{B}^3 \cup_{\Psi_i} H.\]
Here $\Psi_i$ is a homeomorphism from a tubular neighborhood $T_i$ of $K_i$ in $\partial \mathbb{B}^4$ to $T.$ 

We define the \emph{core of the pseudo handle} to be the cone of $K_i$ with cone point the center of $\mathbb{B}^4.$ The core is then a 2-disc whose interior lies in $\intr \mathbb{B}^4.$

\begin{figure}[!ht]
\vspace{-5.5in}
\includegraphics[height=7in]{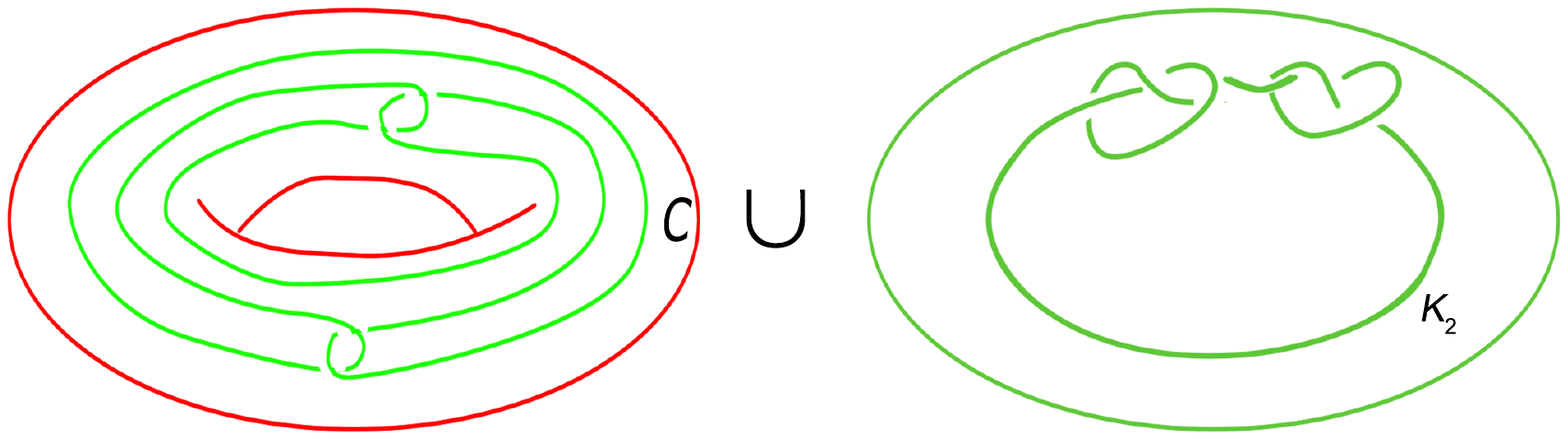} 
\centering
\caption{$S^1 \times \mathbb{B}^3$ union a degree 2 pseudo $2$-handle}
\label{pseudo}
\end{figure}

\begin{prop} Each $W_i \searrow J.$
\end{prop}

\begin{proof} The same proof as for every Jester's manifold collapses to $J$ (Proposition \ref{M collapses to J}) goes through with the pseudo 2-handle collapsing to its core. 
$H$ collapses to its core union its attaching tube defined as $\Psi_i(T_i)$. $M(g)$ again collapses to $M(h)$ with the attaching tube collapsing to the attaching sphere: $\Psi_i(K_i)=C$. 

\end{proof}

\begin{cor} Each $W_i=\mathbb{B}^4 \cup_{\mathbb{B}^4} \mathbb{B}^4.$
\end{cor}

\subsection{A Theorem of Wright}
\label{Theorem of Wright}

Applying the following theorem will yield an infinite collection of distinct $W_i$. Before we state the theorem we'll need some definitions.

\begin{defn}

A $3$-manifold is \emph{irreducible} if every embedded $S^2$ bounds a $\mathbb{B}^3.$

\end{defn}

\begin{defn}

A torus $S$ in a $3$-manifold $X$ is said to be \emph{incompressible in $X$} if the homomorphism induced by inclusion $\pi_1(S) \rightarrow \pi_1(X)$ is injective.

\end{defn}

\begin{defn}

A group $G$ is \emph{indecomposable} if for all subgroups $A, B$ such that ${G \approx A\ast B}$, either $A=1$ or $B=1$. (That is, $G$ contains no nontrivial free factors.)

\end{defn}

\begin{thm} 
\label{Wright}
\cite{Wri} Suppose $X$ is a compact $4$-manifold obtained from the $4$-manifold $N$ by adding a $2$-handle $H$. If \emph{cl}$(\partial X-H)$ is an orientable irreducible $3$-manifold with incompressible boundary, then there exists a countably infinite collection of compact $4$-manifolds $M_i$ such that\\
(1) $\pi_1 (\partial M_i) \ncong \mathbb{Z}$ and is indecomposable\\
(2) $\pi_1(\partial M_i) \ncong \pi_1(\partial M_j)$ for $i\neq j$ and hence, \emph{int}$(M_i)$ is not homeomorphic to \emph{int}$(M_j).$
\end{thm}

Wright constructs the infinite collection of manifolds $\{M_i\}$ of the theorem as follows. For each $i=1,2,...$ he constructs a manifold by attaching to $N$ a psuedo 2-handle along $K_i$. From this sequence he exhibits a subsequence $\{M_{i_j}\}$ each term of which has a distict boundary.

For the proof of the following theorem we'll employ the Loop Theorem \cite[p.~101]{Rol}.

\begin{thm}\emph{(Loop Theorem)} If $X$ is a 3-manifold with boundary and the induced inclusion homomorphism $\pi_1(\partial X) \rightarrow \pi_1(X)$ has  nontrivial kernel, then there exists an embedding of a disc $D$ in $X$ such that $\partial D$ lies in $\partial X,$ and represents a nontrivial element of $\pi_1(\partial X)$. 
\end{thm}

\begin{thm} 
There exists an infinite collection of closed 4-dimensional splitters. The fundamental groups of their boundaries are distinct, indecomposable, and noncyclic. 

\label{infinite 4-splitters}
\end{thm}

\begin{proof}

We'll show $M$ meets the hypotheses of Theorem \ref{Wright}, thus yielding a subsequence of $\{W_i\}$ as our desired collection. Recall $T$ is the tubular neighborhood of the attaching sphere $C$ in the construction of the Jester's manifold so that $\partial T = \partial\text{cl}(\partial M - h^{(2)})$. It suffices to show

\begin{claim} 
\label{bdry T incompressible}
$\partial T$ is incompressible in cl$(\partial M - h^{(2)})= S^1 \times S^2- \text{int}(T).$
\end{claim}

We will show $\text{ker}(\pi_1(\partial T) \rightarrow \pi_1(S^1 \times S^2 - \text{int}(T)))=1.$ Recall $T_\Gamma$ is the tubular neighborhood of the Mazur curve $\Gamma$ in the $S^1 \times S^2$ in the construction of the Mazur manifold $M\!a^4$ (see Subsection \ref{sec:MazMan}). Recall further Proposition \ref{gamma nontrivial}:
Let $m_\Gamma$ be the meridian of the torus $\partial T_\Gamma.$ Then $m_\Gamma$ is nontrivial in $S^1 \times S^2 - \text{int}(T_\Gamma).$

By construction $
S^1 \times S^2 - \intr(T)$ 
is a double cover of $S^1 \times S^2 - \intr(T_\Gamma).$ Call the associated covering map $p$ and let $m$ be a lift of $m_\Gamma$ so $m$ is a meridian of $\partial T$. Then $p_*([m])=[m_\Gamma]\neq 1$ gives $[m] \neq 1.$ Suppose by way of contradiction that there exists an embedded disc $D$ in $S^1 \times S^2 - \intr(T)$ with $\partial D$ being a nontrivial loop in $\partial T.$ 
Choose a longitude $l$ on $\partial T$ and let $\mu= [m]$ and $\lambda=[l]$ in $\pi_1(\partial T)$ so that for some $k,j \in \mb{Z},$ $[\partial D] = \mu^k \lambda^j$ in $\pi_1(\partial T).$ 
As $C$ has algebraic index 1 in $S^1 \times S^2$ a nonzero $j$ would imply $[\partial D]$ nontrivial in $\pi_1(S^1 \times S^2 - \intr(T)).$ Thus $[\partial D]=\mu^k.$ But any loop going around meridianally more than once and longitudinally zero will not be embedded. See \cite[p.~25]{Spa} for an illustration. 
Then it must be that $[\partial D]=[m]^{\pm1}.$ Since $m$ is nontrivial in $S^1 \times S^2 - \intr T$ such a $D$ cannot exist and by the Loop Theorem $\text{ker}(\pi_1(\partial T) \rightarrow \pi_1(S^1 \times S^2 - \text{int}(T)))=1.$ 
\end{proof}

\begin{defn}
We call any $M_i$ as yielded by the theorem when applied to any $M_\Psi$ a \emph{Jester's manifold}. 
\end{defn}

Note that for a given knot $K_i$, different choices of framing homeomorphism potentially yield different manifolds. So the variety of distinct Jester's manifolds produced by this construction is potentially much greater than we have shown.

We conclude this section with a restatement of our first main result which we have now demonstrated.\\

\noindent
\textbf{Theorem \ref{infinite_closed_splitters}.} \emph{
There exists an infinite collection of topologically distinct splittable compact contractible 4-manifolds. The interiors of these are topologically distinct contractible splittable open 4-manifolds.}

\section{Sums of Splitters}

In this our concluding section, we will exhibit an uncountable collection of contractible open 4-dimensional splitters. We will do so by considering the interiors of infinite boundary connected sums of our Jester's manifolds. These open manifolds can also be constructed as the connected sum at infinity of the interiors of the same sequence of manifolds. Using the notion of the fundamental group at infinity we will be able to show that any two such sums where one Jester's manifold appears more often as a summand in one than the other are topologically distinct. We then demonstrate a splitting for such manifolds.

\subsection{Some Manifold Sums and the Fundamental Group at Infinity}

We describe what we mean by the \emph{induced orientation} of the boundary of an oriented manifold $X^n$. Given a collar neighborhood of $\partial X$ which we identify as $\partial X \times [0,1]$ $(\partial X$ identified with $\partial X \times \{0\} )$ and a map $h:\mathbb{B}^{n-1} \rightarrow \partial X$ we define $\bar h$ as \[\bar{h}:\mathbb{B}^n \rightarrow \partial X \times (0,1], \ \ \ \bar{h}(x_1,x_2,...,x_n)= \left(h(x_1,x_2,...,x_{n-1}), \frac{3+x_n}{4}\right).\] (To be precise the codomain of $\bar{h}$ should be $\intr X.$) 
If $h:\mathbb{B}^{n-1} \rightarrow \partial X$ and $\bar{h}$ is a representative of the orientation of $X$ then the ambient isotopy class of $h$ is the \emph{induced orientation} of $\partial X$. \cite[p.~45]{RoSa}.

\begin{defn} Let $M^n$ and $N^n$ be connected oriented manifolds with nonempty boundaries. Orient $\Bd M$ and $\Bd N$ with their induced orientations and let $B_M$ and $B_N$ be tame $(n-1)$-balls in $\partial M^n$ and $\partial N^n$, respectively. Let $\phi:B_M \rightarrow B_N$ be an orientation reversing homeomorphism. Then $M^n \cup_\phi N^n$ is called a \emph{boundary connected sum} (\emph{BCS}) and is denoted $M^n \bdrysum N^n.$ 
\end{defn}

\begin{prop}
\label{bdry sum well defined}
The boundary connected sum is a connected oriented manifold which, provided $\Bd M$ and $\Bd N$ are connected, does not depend on the choices of $B_i$ or $\phi_i.$ Furthermore the set of connected oriented $n$-dimensional manifolds with connected boundaries is, under the operation of connected sum, a commutative monoid (that is, associative and contains an identity) the identity being $\mathbb{B}^n$ \cite[p.~97]{Kos}. 
\end{prop}

\begin{defn} Let $\{M^n_i\}_{i=1}^m$ ($m$ possibly $\infty$) be oriented manifolds with  \\ nonempty connected boundaries and for each $i=1,2,...$ let $B_{i,L}$ and $B_{i,R}$ be disjoint tame $(n-1)$-balls in $\partial M^n_i$. For $i>1$ let $\phi_i:B_{i,L} \rightarrow B_{i-1,R}$ be an orientation reversing homeomorphism. Let $\phi:\sqcup_{i>1} B_{i,L} \rightarrow \sqcup_{i\geq 1} B_{i,R}$ with $\phi|_{B_{i,L}}=\phi_i.$ Then $\left(\sqcup M_i\right)/\phi$ is called a \emph{boundary connected sum} (\emph{BCS}) and is denoted $M_1 \bdrysum M_2 \bdrysum \cdots \bdrysum M_m$ (or $M_1 \bdrysum M_2 \bdrysum \cdots$ when $m=\infty).$  
\end{defn}

We next prepare a description of an analogous sum for open manifolds. But first we need a proposition ensuring the existence of the desired attaching maps. 

\begin{defn} By a \emph{proper map} $p$ between spaces $Y$ and $X$ we mean a map $p:Y \rightarrow X$ such that for any compact $C \subset X$ we have $p^{-1}(C)$ is compact. 
A \emph{ray} is a proper embedding $[0,\infty)\rightarrow X.$
\end{defn}

\begin{note} Unless otherwise stated all rays will be piecewise linearly embedded. We will abuse our notation for rays (as well as for some other maps) by using our symbol for the map to also mean its image.

\end{note}

\begin{prop}
\label{nhd of proper ray}
 Suppose $N$ is a regular neighborhood of a ray $r$ in an open $n$-manifold $M (n\geq 4).$ Then $(N, \partial N) \approx (\mathbb{R}^n_+,\mathbb{R}^{n-1}).$
\end{prop}

A proof of this proposition can be found in \cite[p.~28]{Spa}.

\begin{defn} For oriented, piecewise linear, open $n$-manifolds $X$ and $Y$, and rays $\alpha_X \subset X$ and $\alpha_Y \subset Y$ we define the \emph{connected sum at infinity (CSI)} of $(X,\alpha_X)$ and $(Y,\alpha_Y)$ as follows. Choose regular neighborhoods $N_X$ and $N_Y$ of $\alpha_X$ and $\alpha_Y$, respectively. Orient $\partial N_X$  with the induced orientation from the given orientation of $X-\intr N_X$ and orient $\partial N_Y$ from the given orientation of $Y-\intr N_Y.$ Then the CSI of $(X,\alpha_X)$ and $(Y,\alpha_Y)$ is
\[(X,\alpha_X)\natural (Y,\alpha_Y)= (X-\text{int}N_X)\cup_f (Y-\text{int}N_Y)\]
where $f$ is an orientation reversing p.l. homeomorphism $f: \partial N_X \rightarrow \partial N_Y$. 
\end{defn}

\begin{figure}[!ht]
\vspace{-4in}
\includegraphics[height=6in]{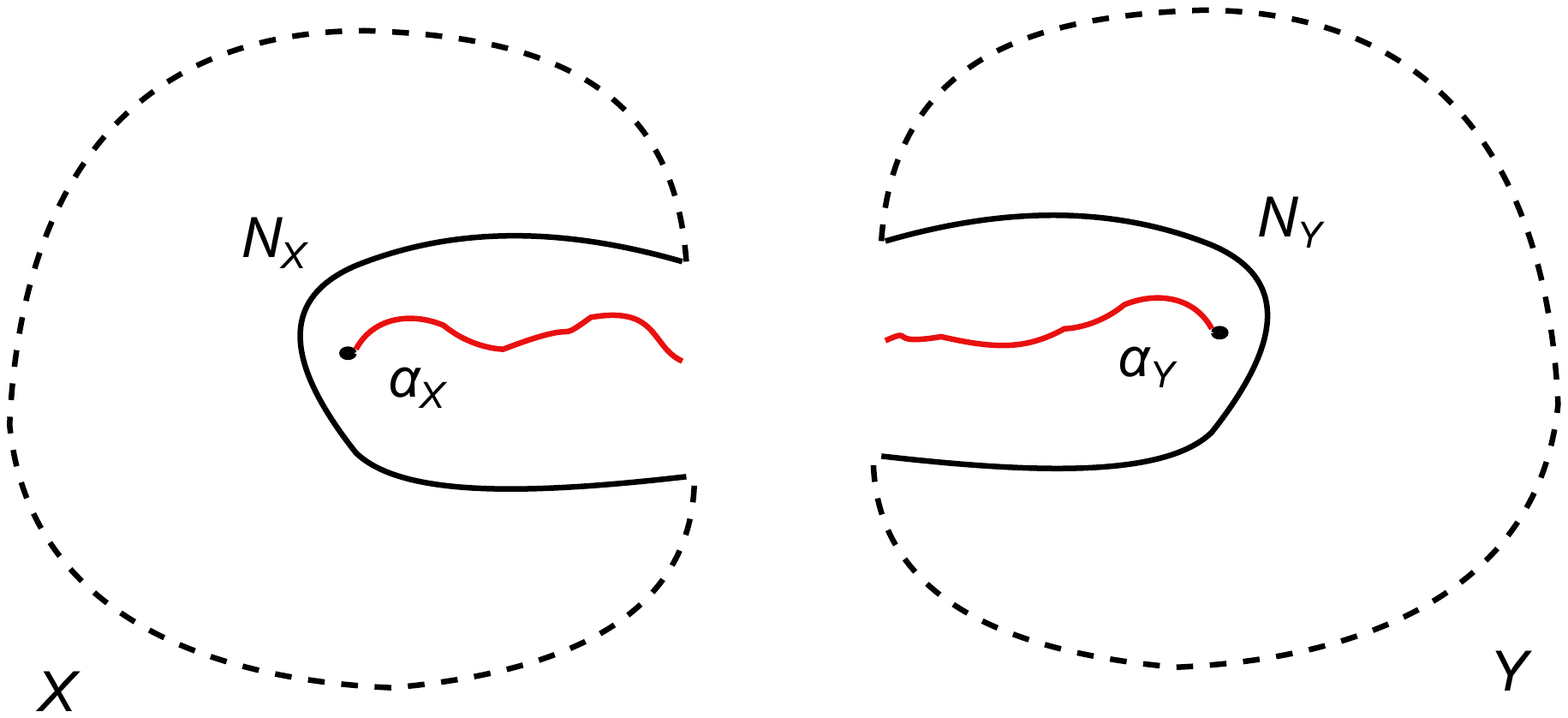} 
\centering                 
\caption{$(X,\alpha_X)\natural (Y,\alpha_Y)$}
\end{figure}

Note that we are considering regular neighborhoods of \emph{noncompact} manifolds and by the uniqueness theorem for regular neighborhoods (see \cite[p.~196]{Coh69})\\
 $(X,\alpha_X)\natural (Y,\alpha_Y)$ is independent of the choices of neighborhoods $N_X$ and $N_Y.$

We note that (for our conditions on the summands) our definition of $X \natural Y$ is equivalent to both Gompf's definition of \textit{end sum} \cite{Gom} and Calcut, King, and Siebenmann's definition of connected sum at infinity \cite{CKS}. 

\begin{defn} Let $\{X_i\}_{i=1}^m$ ($m$ possibly $\infty),$ be oriented, piecewise linear, open $n$-manifolds and for $i=1,2,...$ and $x=L,R$ choose disjoint rays $\alpha_{i,x} \subset X_i.$ Further choose regular neighborhoods $N_{i,x}$ of $\alpha_{i,x}$ so that $N_{i,L}\cap N_{i,R}=\emptyset.$ Orient $\partial N_{i,x}$ with the induced orientation from the given orientation of $X_i-\intr(N_{i,L}\cup N_{i,R})$ and choose orientation reversing homeomorphisms ${\phi_i:\partial N_{i,R} \rightarrow \partial N_{i+1,L}.}$ Let $\phi: \sqcup_{i\geq 1} \partial N_{i,R} \rightarrow \sqcup_{i> 1} \partial N_{i,L}$ with $\phi|_{N_{i,R}}=\phi_i.$ Let $\check{X}_1=X_1-\intr N_{1,R}$ and for $i=2,3...$ let $\check{X}_i= X_i-\intr(N_{i,L} \cup N_{i,R}).$ Then ${(\sqcup \check{X_i})/\phi}$ is called the \emph{connected sum at infinity} (\emph{CSI}) of $\{(X_i,\alpha_{i,L}, \alpha_{i,R})\}.$ We denote this sum as $(X_1,\alpha_{1,L},\alpha_{1,R})\natural ... \natural(X_m, \alpha_{m,L}, \alpha_{m,R})$ (or $(X_1,\alpha_{1,L},\alpha_{1,R})\natural(X_2,\alpha_{2,L}, \alpha_{2,R})\natural ...$ when $m$ is $\infty.$) 
\end{defn}

\begin{remark}
A flexible theory of connected sum at infinity is presented in \cite{CKS}. Among other things the order of the summands does not effect the homeomorphism type of the resulting manifold.
\end{remark}

\begin{remark}

The connected sum at infinity of the interiors of manifolds with connected boundary is homeomorphic to the interior of their boundary connected sum.  
For a CSI of open manifolds which are not the interiors of compact manifolds (Whitehead's exotic open 3-manifold or Davis manifolds, for example \cite[p.~6]{Gui}) we do not have the luxury of utilizing this result.  
\end{remark}

We'll now prepare the definition of the fundamental group at infinity of a 1-ended topological space. This is an invariant of spaces which are 1-ended and satisfy the condition that any pair of proper rays can be joined by a proper homotopy. (See \cite{Gui} for a much more thorough treatment of this topic.) Let $\{G_j,\varphi_j\}$ be an inverse sequence of groups:
\begin{diagram}
G_1 &  \lTo^{\varphi_2} &
G_2 & \lTo^{\varphi_3}  &
G_3 & \lTo^{\varphi_4}  &
\cdots .
\end{diagram}
For an increasing sequence of positive integers $\{j_i\}_{i=1}^\infty$, let \[f_i=\varphi_{j_{i+1}}...\varphi_{j_i+1}\varphi_{j_i}:G_{j_i}\rightarrow G_{j_{i+1}}\] 
and call the inverse sequence $\{G_{j_i},f_i\}$ a \emph{subsequence} of the inverse sequence $\{G_j,\varphi_j\}.$

We say the inverse sequences $\{G_j,\varphi_j\}$ and  $\{H_k,\psi_k\}$ are \emph{pro-isomorphic} if there exists subsequences $\{G_{j_i},f_i\}$ and $\{H_{j_i},g_i\}$ that may be fit into a commuting ladder diagram of the form

\begin{diagram}
G_{j_{1}} & & \lTo^{f_{2}} & &
G_{j_{2}} & & \lTo^{f_{3}} & &
G_{j_{3}} & & \lTo^{f_{4}} & &
\cdots\\
& \luTo^{u_{1}} & & \ldTo^{d_{2}} & & \luTo^{u_{2}} & & \ldTo^{d_{3}}  & & \luTo^{u_{3}} \\
& & H_{k_1} & & \lTo^{g_{2}} & & H_{k_{2}}
& & \lTo^{g_{3}}  & & H_{k_{3}} & & \lTo^{g_{4}} & \cdots 
\end{diagram}
Pro-isomorphism is an equivalence relation on the set of inverse sequences of groups.

\begin{defn} We say the inverse sequence of groups $\{G_j,\varphi_j\}$ is 
\emph{stable} if it is pro-isomorphic to a constant sequence $\{H,\text{id}_H\}$, and we say $\{G_j,\varphi_j\}$ is
\emph{semistable} if it is pro-isomorphic to an $\{H_k,\psi_k\},$ where each $\psi_k$ is an epimorphism.
 
\end{defn}

We call $A\subset X$ a \emph{bounded set} (in $X$) if $\cl(X-A)$ is compact. We define a \emph{neighborhood of infinity} of  a topological space $X$ to be the complement of a bounded subset of $X.$ A \emph{closed (open) neighborhood of infinity} in $X$ is one that is closed (open) as a subset of $X.$ A closed neighborhood of infinity $N$ of a manifold $M$ with compact boundary is \emph{clean} if it is a codimension 0 submanifold disjoint from $\partial M$ and $\partial N=\Bd_M N$ has a bicollared neighborhood in $M.$ Here we are using the notation $\Bd_M N$ in the following sense. For $A$ a subset of a topological space $Z,$ $\Bd_Z A$ will denote the (topological) boundary (also known as the frontier) of $A$ in $Z$ (not to be confused with the notion of manifold boundary). We say $X$ is \emph{k-ended} if $k<\infty$ and $k$ is the least upper bound of the set of cardinalities of unbounded components of neighborhoods of infinity of $X.$ That is,
\[k=\sup\{|\{\text{unbounded components of }N\}|:N \text{ a neighborhood of infinity of }X\}.\]
In the case the above supremum is infinite, we say $X$ is \emph{infinite ended}.

By a \emph{cofinal} sequence $\{U_j\}$ of subsets of $X$ we mean $U_j \supset U_{j+1}$ and $\bigcap U_j= \emptyset.$ Now let $X$ be a 1-ended space and choose a cofinal sequence $\{U_j\}$ of connected neighborhoods of infinity of $X.$ Choose a ray (called a \emph{base ray}) $r$ in $X$ and base points $x_j \in r \cap U_j$ such that $r([r^{-1}(x_j),\infty)) \subset U_j.$ Let $G_j=\pi_1(U_j,x_j)$ and $\tau_j:G_j \rightarrow G_{j-1}$ be the homomorphism (called a \emph{bonding homomorphism}) defined as follows. Let $\iota_j: \pi_1(U_j,x_j)\rightarrow \pi_1(U_{j-1},x_j)$ be the homomorphism induced by the inclusion $U_j\hookrightarrow U_{j-1}$ and $\rho_j$ be the canonical basepoint change isomorphism. This isomorphism is induced by the map that generates a loop $\alpha'$ based at $x_{j-1}$ from a loop $\alpha$ at $x_{j}$ by starting at $x_{j-1}$ following $r$ to $x_{j}$ traversing $\alpha$ and returning along $r$ to $x_{j-1}.$ Then $\tau_j$ is defined as $\tau_j=\rho_j \circ \iota_j$ and $\{G_j,\tau_j\}$ is a inverse sequence of groups. We then define the \emph{fundamental group of infinity (based at r)} of $X$ (denoted $\text{pro-}\pi_1(\epsilon(X),r))$ to be the pro-isomorphism class of $\{G_j,\varphi_j\}.$ It can be shown that this class is independent of the choice of $\{U_j\}.$

The following theorem can be found in \cite[pp.~29-31]{Gui}.
\begin{thm}Let $X$ be a $1$-ended space.  If $\text{pro-}\pi_1(\epsilon(X),s)$ is semistable for some ray $s$ then any two rays in $X$ are properly homotopic and conversely. Further in any such space $\text{pro-}\pi_1(\epsilon(X),r)$ is independent of base ray $r.$
\label{semistable}
\end{thm}

We call any 1-ended manifold $X$ that meets either of the equivalent conditions of Theorem \ref{semistable} \emph{semistable}. A \emph{stable} 1-ended manifold $X$ is one for which \\
$\text{pro-}\pi_1(\epsilon(X),r)$ is stable (hence semistable and thus independent of $r$). If $M$ is a compact manifold with connected boundary (for example any Jester's manifold) then the interior of $M$ is 1-ended and stable \cite[p.~33]{Spa}.

\subsection{CSI's of Semistable Manifolds}

We'll next show that the CSI of a collection of semistable manifolds is independent of the choice of rays.

\begin{defn}
We say $N$ is a \emph{half space} of a manifold $M^n$ if $N$ is the image of an embedding $h:\mathbb{R}_+^n\rightarrow M.$
We say such an $N$ is a \emph{proper} half space if the embedding $h$ is proper. We say $N$ is a \emph{tame} half space if $h(\partial \mathbb{R}^n_+)$ is bicollared in $M^n.$
\end{defn}

\begin{prop} 
\label{open minus half sp}
If $M$ is an open, contractible manifold and $N$ is a proper and tame half space of $M$, then $M-N \approx M.$
\end{prop}
See \cite[pp.~33-34]{Spa} for a proof.

\begin{note}
Proposition \ref{nhd of proper ray} along with Proposition \ref{open minus half sp} imply when 
\[(X,\alpha_X)\approx(Y,\alpha_Y) \approx (\mb{R}^n,\{0\}^{n-1}\times [0,\infty))\]
 we have $(X,\alpha_X)\natural (Y,\alpha_Y)\approx \mb{R}^n.$
\label{csi of nspaces} 
\end{note}

\begin{lemma} Suppose $M^n$ $(n \geq 4)$ is a contractible, oriented, piecewise linear, semistable, open manifold. If $r$ and $r'$ are PL rays in $M^n$ and $N$ and $N'$ are regular neighborhoods of $r$ and $r'$, respectively, then there exists an orientation preserving self homeomorphism of $M^n$ taking $N$ to $N'$. 
\label{indep of rays}
\end{lemma}

A proof of this lemma can be found in \cite[pp.~34-37]{Spa}.

\begin{cor} For $n \geq 4$ and 1-ended semistable manifolds $X^n,Y^n,X_1^n, X_2^n,...$  $(X,\alpha_X) \natural (Y,\alpha_Y),$ $(X_1,\alpha_{1,L},\alpha_{1,R})\natural...\natural (X_m,\alpha_{m,L}, \alpha_{m,R})$ and \\
${(X_1,\alpha_{1,L},\alpha_{1,R})\natural (X_2,\alpha_{2,L}, \alpha_{2,R})\natural...}$ are independent of choices of rays $\alpha_X$, $\alpha_Y,\alpha_{1,L},$\\
$\alpha_{1,R},\alpha_{2,L}, \alpha_{2,R},....$
\end{cor}

As a result of the corollary, when considering 1-ended semistable $n$-manifolds $(n\geq 4)$ $X,$ $Y,$ $X_1,X_2,...$ we will use the notations $X \natural Y,$ $X_1\natural...\natural X_m,$ and $X_1\natural X_2 \natural ...$  for the unique CSI's of $X$ and $Y,$ $X_1,X_2,...,X_m,$ and $X_1, X_2,....$

The following proposition can be justified by an application of Van Kampen's Theorem. 
\begin{prop}
Let $X$ and $Y$ be 1-ended semistable open $n$-manifolds $(n\geq 4).$ Then $\pi_1(X\natural Y)\cong \pi_1(X)*\pi_1(Y).$
\label{pi_1 CSI}
\end{prop}

\subsection{Some Combinatorial Group Theory and Uncountable Jester's Manifold Sums}

The primary goal of this subsection is the following. 

\begin{thm}
The set of homeomorphism classes of all possible CSI's of interiors of Jester's manifolds is uncountable.
\label{uncountable Jester sums}
\end{thm}

This Theorem can be obtained from Theorem \ref{infinite 4-splitters} by an application of Curtis and Kwun's Theorem (4.1)
 of \cite{CuKw}. Since the approach used there is a bit outdated, we
will supply an alternate version of their theorem. 
The essence of our proof is the same as theirs,
but ours will take advantage of the rigorous development of the
fundamental group at infinity that has
taken place in the intervening years. The new approach is also more direct in
that it compares open manifolds directly, without reference to some discarded
boundaries. We will demonstrate shortly the following more general result, for which Theorem \ref{uncountable Jester sums} will be a corollary.

\begin{thm}
Let $\mathcal{G}$ be a collection of distinct
indecomposable groups, none of which are infinite cyclic and let $\left\{  X_{i}^{n}\right\}  $and $\left\{
Y_{j}^{n}\right\}  $ be countably infinite collections of simply connected, 1-ended
open $n$-manifolds with each $\operatorname*{pro}$-$\pi_{1}\left(
\varepsilon\left(  X_{i}\right)  \right)  $ and $\operatorname*{pro}$-$\pi
_{1}\left(  \varepsilon\left(  Y_{j}\right)  \right)  $ being stable and
pro-isomorphic to an element of $\mathcal{G}$. Then $X_1 \natural X_2 \natural ...$ and $Y_1 \natural Y_2 \natural...$ are 1-ended and semistable and if any element of $\mathcal{G}$
appears more times in one of the sequences, $\left\{  \operatorname*{pro}%
\text{-}\pi_{1}\left(  \varepsilon\left(  X_{i}\right)  \right)  \right\}  $
and $\left\{  \operatorname*{pro}\text{-}\pi_{1}\left(  \varepsilon\left(
Y_{j}\right)  \right)  \right\}  $, than it does in the other, then
$\operatorname*{pro}$-$\pi_{1}\left(  \varepsilon\left(  X_{1}\overset{}{\natural}X_{2}\overset{}{\natural}X_{3}\overset{}{\natural}\cdots\right)  \right)  $is not pro-isomorphic to $\operatorname*{pro}$-$\pi_{1}\left(  \varepsilon\left(  Y_{1}\overset{}{\natural}Y_{2}\overset{}{\natural}Y_{3}\overset{}{\natural}\cdots\right)  \right)  .$\bigskip
\label{uncountable CSI's}
\end{thm}

First we'll state and prove a theorem about certain types of inverse sequences of groups that will help us determine when two infinite CSI's of our Jester's manifolds are distinct. This theorem (or its discovery) and its proof are motivated by Theorem (4.1) (and its proof) in \cite{CuKw}. 

\begin{thm}
Let $A_1,A_2,...$ and $B_1,B_2,...$ be indecomposable groups none of which are infinite cyclic, and for all positive integers $j$ and $k$ let $G_j$ and $H_k$ be the free products 
\[G_j=A_1\ast A_2\ast ... \ast A_j\] \[H_k=B_1\ast B_2 \ast... \ast B_k.\] Further let 
$\varphi_j:G_j\rightarrow G_{j-1}$ and
$\psi_k:H_k\rightarrow H_{k-1}$
be the obvious projections 
such that \[\varphi_j|_{G_{j-1}}=id_{G_{j-1}}\mbox{,             } \varphi_j(A_j)=1,\] \[\psi_k|_{H_{k-1}}=id_{H_{k-1}} \mbox {   and,   }\psi_k(B_k)=1.\]
Suppose the inverse sequences $\{G_j,\varphi_j\}$ and $\{H_k, \psi_k\}$ are pro-isomorphic. That is, there exists a commutative ladder diagram as below.

\[
\begin{diagram}
G_{j_{1}} & & \lTo^{f_{2}} & &
G_{j_{2}} & & \lTo^{f_{3}} & &
G_{j_{3}} & & \lTo^{f_{4}} & &
\cdots\\
& \luTo^{u_{1}} & & \ldTo^{d_{2}} & & \luTo^{u_{2}} & & \ldTo^{d_{3}}  & & \luTo^{u_{3}} \\
& & H_{k_1} & & \lTo^{g_{2}} & & H_{k_{2}}
& & \lTo^{g_{3}}  & & H_{k_{3}} & & \lTo^{g_{4}} & \cdots
\end{diagram}
\]
\\
Here the bonding homomorphisms are the compositions
\[f_i=\varphi_{j_{i+1}}...\varphi_{j_i+1}\varphi_{j_i}
\mbox{ and } 
g_i=\psi_{k_{i+1}}...\psi_{k_i+1}\psi_{k_i}.\] Then there exists a self bijection $\Phi$ of $\mathbb{Z}_+$
such that $A_j\cong B_{\Phi(j)}.$
\label{CuKw type}
\end{thm}
 
\begin{proof}
It suffices to show the following two claims.

{\bf Claim 1}: For each positive integer pair $(l,s)$ with $l \leq j_s$ and $s>1$ 
there exists at least as many isomorphic copies of $A_l$ among $B_1,..., B_{k_s}$ as there are among $A_1,..., A_{j_s}$.

{\bf Claim 2}: For each positive integer pair $(r,m)$ with $r \leq k_m$  
there exists at least as many isomorphic copies of $B_r$ among $A_1,..., A_{l_m}$ as there are among $B_1,..., B_{k_m}$.

We prove claim 1 and by a similar argument one can prove claim 2. We will use the following facts: in a group $C=C_1 \ast C_2 \ast ...\ast C_n$ (1) no nontrivial free factor $C_i$ is a subgroup of a conjugate of some other free factor $C_j$ and (2) every conjugate of $C_i$ meets every other factor $C_j$, $j\neq i$ trivially. These facts can be verified using normal forms  
\cite[p.~175]{LySc}.
Consider the following commutative ladder diagram.

\[
\begin{diagram}
G_{j_{s}} & & \lTo^{f_{s+1}} & &
G_{j_{s+1}} & & \lTo^{f_{s+2}} & &
G_{j_{s+2}} & & \lTo^{f_{s+3}} & &
\cdots\\
& \luTo^{u_{s}} & & \ldTo^{d_{s+1}} & & \luTo^{u_{s+1}} & & \ldTo^{d_{s+2}}  & & \luTo^{u_{s+2}} \\
& & H_{k_s} & & \lTo^{g_{s+1}} & & H_{k_{s+1}}
& & \lTo^{g_{s+2}}  & & H_{k_{s+2}} & & \lTo^{g_{s+3}} & \cdots
\end{diagram}
\]

We observe that for $i>1$,  $d_i|_{G_{j_{i-1}}}:=d_i\circ (G_{j_{i-1}} \hookrightarrow G_{j_i})$ and $u_i|_{H_{k_{i-1}}}$ are monomorphisms since $f_i|_{G_{j_{i-1}}}$ and $g_i|_{H_{k_{i-1}}}$ are. Thus $A_i \cong d_{s+1}(A_i)$ and $B_k \cong u_{s+1}(B_k)$ for $i \leq j_s$ and $k \leq k_s$. 

Choose $l\leq j_s.$  We'll show there exists $t\leq k_{s}$ such that $u_{s+1}(B_t)$ is a conjugate of $A_l$ thus exhibiting $B_t$ as an isomorphic copy of $A_l.$ Since $d_{s+2}(A_l) \cong A_l$ is indecomposable and not infinite cyclic the Kurosh Subgroup Theorem \cite[p.~ 219]{Mas} gives $d_{s+2}(A_l) \leq \beta B_t \beta^{-1}$ for some $t \leq k_{s+1}$ and $\beta \in H_{k_{s+1}}.$ Moreover, since $A_l$ survives into $G_{j_s}$ we have $t\leq k_s.$ Then the restriction $u_{s+1}|_{B_t}$ is injective and thus so is $u_{s+1}|_{\beta B_t \beta^{-1}}$ and we know $u_{s+1}(\beta B_t \beta^{-1})$ is indecomposable and not infinite cyclic. We again apply Kurosh yielding $u_{s+1}(\beta B_t \beta^{-1})$ is a subgroup of a conjugate of some $A_r.$ Thus in $G_{j_{s+1}}$ we have   $A_l=f_{s+2}(A_l)=u_{s+1}d_{s+2}(A_l)\leq u_{s+1}(\beta B_t \beta^{-1}) \leq$ conjugate of $A_r.$ By our facts $l=r$ and we have $A_l = \beta B_t \beta^{-1} \cong B_t.$ More specifically, $t$ is the unique integer less than or equal to $k_s$ for which $u_{s+1}(B_t)$ is conjugate to $A_l.$

Thus we have shown the map 
\[
\Psi:\{1,2,...,j_s\} \rightarrow \{1,2,...,k_s\} \text{;   }  l\mapsto t
\]
is injective and $B_{\Psi(i)} \cong A_{i}$. This completes the proof of claim 1 and the proof of the proposition.

\end{proof}

We now apply Theorem \ref{CuKw type} to prove Theorem \ref{uncountable CSI's}.

\begin{proof}[Proof of Theorem \ref{uncountable CSI's}.]
Let $A_i$ and $B_j$ be groups such that $\text{pro-}\pi_1(\varepsilon(X_i))$ and \\ $\text{pro-}\pi_1(\varepsilon(Y_j))$ are pro-isomorphic to the constant sequences $\{A_i,id_{A_i}\}$ and $\{B_j, id_{B_j}\}.$ Then the hypothesis ``an element of $\mathcal{G}$
appears more times in one of the sequences, $\left\{  \operatorname*{pro}%
\text{-}\pi_{1}\left(  \varepsilon\left(  X_{i}\right)  \right)  \right\}  $
and $\left\{  \operatorname*{pro}\text{-}\pi_{1}\left(  \varepsilon\left(
Y_{j}\right)  \right)  \right\}  $, than it does in the other," translates as there does not exist the bijection $\Phi$ as in the conclusion of Theorem \ref{CuKw type}. Thus if we can show that $X_1\natural X_2\natural...$ and $Y_1\natural Y_2\natural...$ are 1-ended and semistable and also that $\text{pro-}\pi_1(\varepsilon(X_1\natural X_2\natural...))$ and $\text{pro-}\pi_1(\varepsilon(Y_1\natural Y_2\natural...))$ are of the forms $\{G_j,\varphi_j\}$ and $\{H_k,\psi_k\}$ in the statement of Theorem \ref{CuKw type} we will have the desired result.

For $i=1,2,...$ let $U_{i,1}\supset U_{i,2} \supset ...$ be a cofinal sequence of clean neighborhoods of infinity in $X_i$ so that $\{\pi_1(U_{i,j}), \tau_{i,j}\}\in \mbox{ pro-}\pi_1(\varepsilon(X_i))$ can be fit into a commuting ladder diagram with $\{A_i,id_{A_i}\}$

\begin{equation}
\begin{diagram}
\pi_1(U_{i,1}) & & \lTo^{\tau_{i,2}} & &
\pi_1(U_{i,2}) & & \lTo^{\tau_{i,3}} & &
\pi_1(U_{i,3}) & & \lTo^{\tau_{i,4}} & &
\cdots\\
& \luTo^{u_{i,1}} & & \ldTo^{d_{i,2}} & & \luTo^{u_{i,2}} & & \ldTo^{d_{i,3}}  & & \luTo^{u_{i,3}} \\
& & A_i & & \lTo^{id} & & A_i 
& & \lTo^{id}  & & A_i  & & \lTo^{id} & \cdots
\end{diagram}
\label{Ladder diagram of Xi}
\end{equation}
Here $\tau_{i,j}$ is the bonding homomorphism 
discussed in the definition of the fundamental group at infinity.

As in the definition of $X_1\natural X_2 \natural ...,$ for $i=1,2,...$ choose disjoint rays $r_{i,L},r_{i,R} \subset X_i$ and disjoint regular neighborhoods $N_{i,L},N_{i,R}\subset X_i$ of said rays with the additional property that for each $j$, $r_{i,x}$ meets $\Bd_{X_i} U_{i,j}$ transversely in a single point.

For $i=2,3,...$ and for $j=1,2,...$ let \[\hat{U}_{1,j}=U_{1,j}-\intr N_{1,R}\ 
 \ \text{and }\ \hat{U}_{i,j}=U_{i,j}-\intr (N_{i,L}\cup N_{i,R}).\]

We claim $\pi_1(\hat{U}_{i,j})\cong \pi_1(U_{i,j}).$ For $i,j=1,2,...$ and $x=L,R$ let $N_{i,x,j}= U_{i,j} \cap N_{i,x}$ which is homeomorphic to $r_{i,x}((a,\infty)) \times \mathbb{B}^{n-1}$ for some $a>0$ since $r_{i,x}$ meets $\Bd_{X_i}U_{i,j}$ transversely in a single point. We see that $\hat{U}_{i,j}\cap N_{i,x,j} \approx r_{i,x}((a,\infty))\times S^{n-2}$ which is simply connected as $n \geq 4.$ Thus \[\pi_1(U_{i,j})=\pi_1(\hat{U}_{i,j}\cup N_{i,L,j}\cup N_{i,R,j})\approx \pi_1(\hat{U}_{i,j}).\]

For $i=1,2,..,$ let $\hat{X}_i=X_i-N_{i,L}\approx X_i$ and \begin{eqnarray}
W_i &=& \hat{U}_{1,i}\cup_\phi \hat{U}_{2,i}\cup_\phi... \cup_\phi \hat{U}_{i,i}\cup_\phi \hat{X}_{i+1}\natural X_{i+2} \natural X_{i+3}... \nonumber 
\end{eqnarray}

Observe that $W_1,W_2,...$ form a cofinal sequence of connected neighborhoods of infinity in $X_1\natural X_2 \natural ...$ and thus if $U$ is a neighborhood of infinity in $X_1\natural X_2 \natural ...$ then $U \supset W_i$ for some $i.$ This shows $X_1\natural X_2 \natural ...$ is 1-ended. 
Then as \[\hat{U}_{i,j} \cap \hat{U}_{i+1,j}=\partial N_{i,L,j}=\partial N_{i+1,R,j} \approx (a,\infty)\times S^{n-2},\]
\[\hat{U}_{i,i}\cap_\varphi \hat{X}_{i+1}=\partial N_{i,R,j},\]
and the $X_i$ are all simply connected we have
\[\pi_1(W_j)\cong \pi_1(U_{1,j})\ast \pi_1(U_{2,j}) \ast ... \ast \pi_1(U_{j,j}).\]

We will show $\{\pi_1(W_j),\tau_{1,j}\}$ is pro-isomorphic to $\{G_j,\varphi_j\}.$ For our base ray we choose $r_1$ the chosen base ray for $X_1.$ Let
\[1_{i,j}:\pi_1(U_{i,j})\rightarrow 1,\]
\[d_j'=d_{1,j}\ast d_{2,j} \ast ... \ast d_{j,j-1}*1_{j,j},\]
\[d_j':\pi_1(U_{1,j}) \ast \pi_1(U_{2,j}) \ast...\ast \pi_1(U_{j,j})\rightarrow A_1\ast A_2\ast...\ast A_{j-1},\]
\[u_j'=u_{1,j}\ast u_{2,j} \ast ... \ast u_{j,j-1}*u_{j,j}, \]
\[u_j':\pi_1(U_{1,j}) \ast \pi_1(U_{2,j}) \ast...\ast \pi_1(U_{j,j})\rightarrow A_1\ast A_2\ast...\ast A_{j}, \mbox{ and }\]
\[\tau_j'=\tau_{1,j}\ast\tau_{2,j}\ast...\ast\tau_{j-1,j}\ast1_{j,j}\]
\[ \tau_j':\pi_1(U_{1,j}) \ast \pi_1(U_{2,j}) \ast...\ast \pi_1(U_{j,j})\rightarrow \pi_1(U_{1,j}) \ast \pi_1(U_{2,j}) \ast...\ast \pi_1(U_{j-1,j})\]
where $d_{i,j},$ $u_{i,j},$ and $\tau_{i,j}$ are the ``up",``down", and bonding homomorphisms of the previous ladder diagram (\ref{Ladder diagram of Xi}). We then have the following commutative diagram:
\begin{diagram}
\pi_1(U_{1,1}) & & \lTo^{\tau_2'} & &
\pi_1(U_{1,2})\ast \pi_1(U_{2,2}) & & \lTo^{\tau_3'} & &
\cdots\\
& \luTo^{u_1'} & & \ldTo^{d_2'} & & \luTo^{u_2'}\\ 
& & A_1 & & \lTo^{\varphi_2} & & A_1\ast A_2 
& & \lTo^{\varphi_3}  & & \cdots
\end{diagram}

Thus $X_1\natural X_2 \natural ...$ is semistable and $\{\pi_1(W_j),\tau_j'\}$ is pro-isomorphic to $\{G_j,\varphi_j\}.$ Similarly, one can show pro-$\pi_1(\varepsilon(Y))$ is of the form $\{H_k,\psi_k\}.$
\end{proof}

\begin{thm}
Let $\mathcal{G}$ be a collection of distinct indecomposable groups none of which are infinite cyclic and let $\left\{  C_{i}^{n}\right\}  $and $\left\{
D_{j}^{n}\right\}  $ be countably infinite collections of compact simply connected
$n$-manifolds with connected boundaries that have fundamental groups lying in
$\mathcal{G}$. If any element of $\mathcal{G}$ appears more times in one of
the sequences, $\left\{  \pi_{1}\left(  \partial C_{i}^{n}\right)  \right\}  $
and $\left\{  \pi_{1}\left(  \partial D_{j}^{n}\right)  \right\}  $, than it
does in the other, then
\[
\operatorname*{int}\left(  C_{1}\overset{\partial}{\#}C_{2}\overset{\partial
}{\#}C_{3}\overset{\partial}{\#}\cdots\right)  \not \approx \operatorname*{int}\left(
D_{1}\overset{\partial}{\#}D_{2}\overset{\partial}{\#}D_{3}\overset{\partial
}{\#}\cdots\right)  .
\]
\label{Craig's no.2}
\end{thm}

\begin{proof}
Since $C_i$ and $D_i$ are compact with connected boundaries $X_i=\intr C_i$ and $Y_i=\intr D_i$ are 1-ended and stable and thus meet the hypotheses of Theorem \ref{uncountable CSI's}. Since the CSI's of the interiors are homeomorphic to the interiors of the BCS's we have the desired result.
\end{proof}

Theorem \ref{uncountable Jester sums}, which we repeat below, can now be seen to be a corollary to Theorem \ref{Craig's no.2}.\\
\textbf{Theorem \ref{uncountable Jester sums}.} \emph{The set of homeomorphism classes of all possible infinite CSI's of interiors of Jester's manifolds is uncountable.}

In the next subsection we will show that these manifolds split.

\subsection{Sums of Splitters Split} \label{Sums of Splitters Split}
In this subsection we demonstrate our main result: \\

\noindent
\textbf{Theorem \ref{uncountable open 4-splitters}.} \emph{There exists an uncountable collection of contractible open 4-manifolds which split as $\mathbb{R}^4 \cup_{\mathbb{R}^4} \mathbb{R}^4.$} \\

We'll demonstrate the above result by showing that the infinite CSI $X_1\natural X_2 \natural ...$ of certain types of splitters $X_i\approx \mathbb{R}^n \cup_{\mathbb{R}^n} \mathbb{R}^n$ $(n\geq4)$ also splits. Our argument consists of choosing our ray, regular neighborhood pairs in the definition of the CSI to lie in the intersections (the $C_i$'s) of the splittings $A_i \cup_{C_i} B_i\approx \mathbb{R}^n \cup_{\mathbb{R}^n} \mathbb{R}^n.$ This will yield the CSI to be of the form 
\begin{equation}
(A_1\natural A_2\natural...)\cup_{C_1\natural C_2\natural...} (B_1\natural B_2\natural ...)
\label{CSI splitting}
\end{equation}
which is itself an open splitting. We apply this result to our infinite sums of Jester's manifolds, an uncountable collection. The work comes in showing the existence of the desired ray, regular neighborhood pair mentioned above. We desire, for all $i$, that our ray not only lies in $C_i$ but also that it is proper in both $A_i$ and $B_i$ thus ensuring we obtain a splitting of the form (\ref{CSI splitting}).

\begin{prop}
If $\Sigma$ is a smooth properly embedded line in $\mathbb{R}^n$ and $M^{n-1}$ is a closed smooth submanifold of $\mathbb{R}^n$ intersecting $\Sigma$ transversely then $\left|\Sigma \cap M^{n-1} \right|$ is even. 
\end{prop}

\begin{proof}
As $M^{n-1}$ is a codimension $1$, closed submanifold of Euclidean space, the Jordan-Brouwer separation theorem gives that it has an inside and an outside \cite{Ale}. Since at each intersection point of $\Sigma$ with $M,$ $\Sigma$ meets $M$ transversely,  $\Sigma$ passes from $M$'s inside to $M$'s outside or vice versa. 
\end{proof}

\begin{lemma} Suppose $M^n$ is a contractible $n$-manifold which splits as $M^n=A \cup_{C} B$, $A,B,C \approx \mathbb{R}^n.$ Then there exists a ray $r$ in $C$ which is also proper in both $A$ and $B.$
\end{lemma}

\begin{proof}
We will describe a proof that uses differential topology. Analogous proofs are possible in the PL or topological categories. Let $S=A \cap \Bd _{M^n} (C)$ and $T= B \cap \Bd _{M^n} (C)$ so that $\Bd_{M^n}(C)=S \sqcup T.$
Let $\overline{C}=\cl _{M^n}(C)$ so $\overline{C}=C\cup S \cup T.$ Note $S$ and $T$ are closed in $\overline{C}.$ Let $\alpha=[-1,1]$ be an arc in $\overline{C}$ so that $\alpha \cap S = \{-1\}$ and $\alpha \cap T = \{1\}.$ Choose $N \approx \intr \alpha \times \mathbb{B}^{n-1}$ a tapered product neighborhood of $\intr \alpha$ in $C.$ That is, $\Bd_{M^n} (N)-N= \partial \alpha.$

Let $f:S \cup N \cup T \rightarrow \alpha$ be a retraction so that $f^{-1}(-1)=S$, $f^{-1}(1)=T$ and for $x \in \intr \alpha,$ $f({x} \times \mathbb{B}^{n-1})=\{x\}.$ That is, $f$ collapses $N$ along product lines. Note that for $x \in \intr \alpha,$ $f^{-1}(x)$ intersects $\alpha$ transversely precisely at $x.$ We then apply the Tietze extension theorem to get a retraction $f:\overline{C} \rightarrow \alpha.$ We choose such an $f$ that is smooth. We will now adjust $f$ with the aim that $C$ maps to $\intr \alpha.$ Let $W=f^{-1}([-1,0]) \cup N \cup T$ and $b \in C-W.$ Via Urysohn's Lemma choose $\eta:\overline{C} \rightarrow [0,1]$ such that $\eta^{-1}(0) =W$ and $\eta^{-1}(1)=\{b\}.$ Let $g=f-\eta$ so $g|_W=f|_W.$ If $x \notin W$ then $\eta(x)>0$ and $g(x)=f(x)-\eta(x)<f(x)-0\leq 1.$ Thus $g^{-1}(1)=T.$ Similarly we can adjust $g$ to get, say $h$, so $h^{-1}(1)=T,$ $h^{-1}(-1)=S,$ and $h|_{S\cup N \cup T}=f|_{S\cup N \cup T}.$

Via Sard's Theorem we can choose a regular value $v$ of $h$ in $\intr \alpha$ and let $V$ be the component of $h^{-1}(v)$ containing $v$ \cite[p.~227]{Kos}. We observe that $V$ is a smooth $(n-1)$-submanifold of $C$ without boundary which is closed 
in $C$ and intersects $\alpha$ (transversely) precisely at $v$. 
If $V$ were compact, the previous proposition would yield that the number of intersections of the $C$ properly embedded line $\intr \alpha$ with $V$ would be even. Thus $V$ is noncompact and hence is $C$ unbounded. 
We claim $V$ is embedded properly in $C.$ For suppose $K$ is a compactum in $C$ and let $\iota:V \hookrightarrow C$ be the inclusion map. Then $V\cap K=\iota^{-1}(K)$ is a closed subset of $K$ and is hence compact thus showing $\iota$ is proper. There then exists a ray $r$ in $V$ which is proper in $C.$

We now show $r$ is proper in both $A$ and $B.$ Let $K$ be a compact subset of $A$. We claim the end of $r$ lies outside of $K.$ Again by Sard, there exists $\epsilon_1$ and $\epsilon_2$ sufficiently small so that $-1+\epsilon_1<v<1-\epsilon_2$ are regular values of $h.$ Let $T'=h^{-1}([-1+\epsilon_1,1-\epsilon_2]),$ a closed subset of $C.$ Then $K'=K\cap T'$ is a compact subset of $C.$ Therefore, $r$ eventually stays outside of $K'$. But since $r$ lives in $T',$ when it leaves $K'$ it also leaves $K.$
Thus $r$ is proper in $A$ and a similar argument can be made to show $r$ is proper in $B.$
\end{proof}

Recall Proposition \ref{ints of splitters split} which says that the interior of a closed splitter is an open splitter.

\begin{cor} Suppose $M^n$ is a compact contractible $n$-manifold such that \\
${M=A \cup_C B,}$ with $A,B,C \approx \mb{B}^n.$ Then there exists a ray $r$ in \emph{int}$C$ which is also proper in both \emph{int}$A$ and \emph{int}$B.$
\end{cor}

\begin{prop} Let $M_1$ and $M_2$ be contractible, piecewise linear, open  \\ $n$-manifolds $(n\geq4)$ which split as $M_i=A_i \cup_{C_i} B_i$, with $A_i,B_i,C_i \approx \mathbb{R}^n.$ Further let $r_i\subset C_i$ be a ray in $C_i$ which is also proper in both $A_i$ and $B_i.$  Then the connected sum at infinity of $(M_1,r_1)$ and $(M_2,r_2)$ also splits: $(M_1,r_1)\natural (M_2,r_2) =A \cup_{C} B$ with $A,B,C \approx \mathbb{R}^n$.
\label{CSI of splitters}
\end{prop}

An immediate corollary is:

\begin{cor} Let $M_1$ and $M_2$ be contractible, piecewise linear, semistable, open $n$-manifolds $(n\geq4)$ which split as $M_i=A_i \cup_{C_i} B_i$, $A_i,B_i,C_i \approx \mathbb{R}^n.$  Then the connected sum at infinity of $M_1$ and $M_2$ also splits: $M_1\natural M_2 =A \cup_{C} B$ with $A,B,C \approx \mathbb{R}^n$.
\end{cor}

\begin{proof}[Proof of Proposition \ref{CSI of splitters}] For $i=1,2,$ let $N_i$ be a ($A_i, B_i,$ and $C_i$) regular neighborhood of $r_i$. For $X_i = M_i,A_i,B_i,C_i$, let $\hat{X}_i = X_i-\intr(N_i).$ Given an orientation reversing homeomorphism $f:\partial N_1 \rightarrow \partial N_2$ we have $(M_1,r_1) \natural (M_2,r_2)=\hat{M}_1 \cup_{f} \hat{M}_2.$ Let $A=\hat{A}_1 \cup_{f} \hat{A}_2$ and observe that $A=(A_1,r_1) \natural (A_2,r_2).$ Likewise we let ${B=\hat{B}_1 \cup_f \hat{B}_2=(B_1,r_1) \natural (B_2,r_2)}$ and $C=\hat{C}_1 \cup_f \hat{C}_2=(C_1,r_1) \natural (C_2,r_2)$ and we see that $(M_1,r_1) \natural (M_2,r_2) = A \cup_C B$. From Note \ref{csi of nspaces} we know each of $A,B,$ and $C$ are $\mb{R}^n$'s as they are each the connected sum at infinity of two $\mb{R}^n$'s. 
\end{proof}

\begin{prop} For $i=1,2,...,$ let $M_i$ be a contractible, open $n$-manifold ($n\geq4$) such that $M_i=A_i \cup_{C_i} B_i$ with $A_i,B_i,C_i \approx \mathbb{R}^n$ for all $i.$ Further let $r_{i,L}$ and $r_{i,R}$ be disjoint rays in $C_i$ that are also proper in both $A_i$ and $B_i.$ Then
	\[M:=\natural_{i=1}^\infty (M_i,r_{i,L},r_{i,R}) \approx A\cup_C B \]
with $A,B,C \approx \mb{R}^n.$
\label{infinite csi splits}
\end{prop}

\begin{cor} For $i=1,2,...,$ let $M_i$ be a contractible, semistable, open $n$-manifold ($n\geq4$). If $M_i=A_i \cup_{C_i} B_i$ with $A_i,B_i,C_i \approx \mathbb{R}^n$ for all $i$ then
	\[M:=\natural_{i=1}^\infty M_i \approx A\cup_C B \]
with $A,B,C \approx \mb{R}^n.$
\end{cor}

\begin{proof}[Proof of Proposition \ref{infinite csi splits}]
For $i=1,2,...,$ choose disjoint $A_i, B_i,$ and $C_i$ regular neighborhoods $N_{i,L},$ $N_{i,R}$  of $r_{i,L}$ and $r_{1,R},$ respectively. 
For $j=1,2,...,$ let 
\[ \check{C}_j= (C_1- N_{1R})\cup(C_2-[\intr N_{2L}\cup N_{2R}]) \cup...\cup(C_j-[\intr N_{j,L}\cup N_{j,R}]).\]

Then $\check{C}_j=(\natural _{i=1}^j (C_i,r_i)) -N_{j,R} \approx \mb{R}^n - \mb{R}^n_+ \approx \mb{R}^n$ and $\check{C}_j \subset \check{C}_{j+1}.$ 
Let $C=\cup \check{C}_j$, so that $C$ is an ascending union of $\mathbb{R}^n$'s and thus is itself an $\mathbb{R}^n$ \cite{Bro}.
Let	
	\[ \check{A}_j= (A_1-\intr N_{1R})\cup(A_2-[\intr N_{2L}\cup N_{2R}]) \cup...\cup(A_j-[\intr N_{j,L}\cup N_{j,R}]),\] 
		\[ \check{B}_j= (B_1-\intr N_{1R})\cup(B_2-[\intr N_{2L}\cup N_{2R}]) \cup...\cup(B_j-[\intr N_{j,L}\cup N_{j,R}]),\]
$A= \cup \check{A}_j,$ and $B= \cup \check{B}_j$ so that $A,B \approx \mb{R}^n$ and $M=A\cup_C B$. 
\end{proof}

We have demonstrated that any CSI of interiors of Jester's manifolds splits and thus have demonstrated \\

\noindent
\textbf{Theorem \ref{uncountable open 4-splitters}.} \emph{There exists an uncountable collection of contractible open 4-manifolds which split as $\mathbb{R}^4 \cup_{\mathbb{R}^4} \mathbb{R}^4.$} \\

A result of Ancel and Siebenman states that a Davis manifold generated by $C$ is homeomorphic to the interior of an alternating boundary connected sum $\intr(C\bdrysum -C \bdrysum C \bdrysum -C \bdrysum ...). $ Here $-C$ is a copy of $C$ with the opposite orientation \cite{Gui}. We have now proved

\begin{cor}
There exists (non-$\mathbb{R}^4$) 4-dimensional Davis manifold splitters.
\end{cor}


\begin{thebibliography}{99}

\bibitem[Ale] {Ale}
	J.W. Alexander, 
	\emph{A proof and extension of the Jordan-Brouwer separation theorem},
	Trans. Amer. Math Soc. \textbf{23} (1922), 333-349
	
\bibitem[AG95]{AG95}
	Fredric D. Ancel and Craig R. Guilbault,
	\emph{Compact Contractible $n$-Manifolds Have Arc Spines ${(n\geq5)},$}
	Pacific Jounal of Mathematics, Vol. 168, No. 1, 1995

\bibitem[AGS]{AGS}
	Fredric D. Ancel, Craig R. Guilbault, and Pete Sparks
	\emph{Infinite Boundary Connected Sums With Applications},
	Preprint
	


\bibitem [Bro] {Bro}
	\text{M. Brown}, 
	\emph{A monotone union of open $n$-cells is an open $n$-cell},
	Proc. Am. Math. Soc. 12,
	812-814 (1961)


\bibitem [CKS] {CKS} 
 	J.S. Calcut, H.C. King, and L.C. Siebenmann,
 	\emph{Connected sum at infinity and Cantrell-Stallings hyperplane unknotting}. 
 	Rocky Mountain J. Math. 42 (2012),
 	 1803-1862
	
\bibitem[Coh73] {Coh73}
	\text{M. M. Cohen},
	\emph{A course in simple homotopy.}
	\newblock Springer-Verlag, New York 1973
	
\bibitem[Coh69] {Coh69}
	\text{M. M. Cohen},
	\emph{A general theory of relative regular neighborhoods},
	Trans. Amer. Math. Soc. 136, 
	189�229 (1969)


\bibitem[CuKw]{CuKw} M.L. Curtis and Kyung Whan Kwun
 \newblock \emph{Infinite Sums of Manifolds.}
	\newblock Topology 3 (1965), 31 �- 42 
	


	

	 
\bibitem[Gab]{Gab} David Gabai
 \newblock \emph{The Whitehead manifold is a union of two Euclidean spaces.}
 \newblock Journal of Topology 4 (2011), 529 �- 534


\bibitem[GRW]{GRW} Dennis J. Garity; Du$\check{\text{s}}$an D. Repov$\check{\text{s}}$; David G. Wright
 \newblock \emph{Contractible 3-manifolds and the double 3-space property.}
 \newblock Trans. Amer. Math. Soc.,
	\newblock to appear.

\bibitem[Geo]{Geo} Ross Geoghegan
	\newblock \emph{Topological Methods in Group Theory.}
	\newblock Springer 2008.

\bibitem [Gla65] {Gla65} Leslie Glaser
	\newblock \emph{Contractible complexes in $S_n.$}
	\newblock Proc. Amer. Math. Soc. 16, 1357-1364 (1965)

\bibitem [Gla66] {Gla66} Leslie Glaser
	\newblock \emph{Intersections of combinatorial balls of Euclidean spaces.}
	\newblock Bull. Am. Math. Soc. 72, 68-71 (1966)



\bibitem [Gom] {Gom} 
	R.E. Gompf, 
	\emph{An infinite set of exotic $\mb{R}^4$'s}. 
	J. Differential Geom. 21 (1985),
	283-300

\bibitem [Gui] {Gui}
	C. R. Guilbault
	\emph{Ends, Shapes, and Boundaries in Manifold Topology and Geometric Group Theory}
	arXiv:1210.6741v3 [math.GT]
	22 Jul 2013
 


\bibitem[Hat]{Hat} Hatcher, Allen, \emph{Algebraic Topology}, Cambridge University Press, New York, 2001

\bibitem[Kos]{Kos} Kosinski, Antoni A., \emph{Differential Manifolds}, Academic Press, Inc., San Diego, 1992

\bibitem[LySc]{LySc} Lyndon, Roger C. and Schupp, Paul E., \emph{Combinatorial Group Theory}, Springer-Verlag, Berlin, Heidelberg, New York, 1977.

\bibitem[Mas]{Mas} William S. Massey, \emph{Algebraic Topology: An Introduction}, Harbrace college mathematics series, Harcourt, Brace and World, New York, 1967.     

\bibitem [Maz]{Maz} Mazur, Barry 
	\newblock \emph{A note on some contractible $4$-manifolds.}
	\newblock Ann. of Math. (2) 73 1961 221--228.



\bibitem [Rol] {Rol}	Dale Rolfsen 
	\emph{Knots and links}, 
	Publish or Perish, (1976)
	
\bibitem[RoSa]{RoSa}
	C. P. Rourke and B. J. Sanderson,
	\emph{Introduction to Piecewise-Linear Topology}.
	Springer-Verlag, Berlin Heidelberg New York
	1982

\bibitem [Spa] {Spa}
	Pete Sparks
	\newblock \emph{Contractible $n$-Manifolds and the Double $n$-Space Property,}
	ProQuest LLC, Ann Arbor, MI, 2014, Thesis (Ph.D.)-The University of Wisconsin-Milwaukee. MR 3346990


\bibitem[Whi]{Whi} J.H.C. Whitehead
	\emph{Simplicial Spaces, nuclei and $m$-groups}, Proc. Lond. Mat. Soc. \textbf{45} (1939), 243-327

\bibitem[Wri]{Wri} David G. Wright
 \newblock \emph {On 4-Manifolds Cross $I.$}
 \newblock Proceedings of the American Mathematical Society Vol. 58 (1976), 315 �- 318

\bibitem[Zee]{Zee} E. C. Zeeman, 
 \newblock \emph{On the Dunce Hat}
 \newblock Topology 2 (1963), 341 -- 358

\end{thebibliography}
\end{document}